\documentclass[leqno,11pt]{amsart}
\usepackage{amssymb,amsthm,wasysym}

\usepackage[latin1]{inputenc}
\usepackage[T1]{fontenc}
\usepackage[spanish,english]{babel}

\usepackage{hyperref}
\usepackage[usenames,dvipsnames]{color}

\theoremstyle{plain}

\allowdisplaybreaks

\newtheorem{thm}{Theorem}[section]

\newtheorem{pro}[thm]{Proposition}

\newtheorem{lem}[thm]{Lemma}
\newtheorem{rem}[thm]{Remark}
\newtheorem*{ex*}{Example}

\numberwithin{equation}{section}

\newcommand{\B}{\mathbb{B}}
\newcommand{\G}{\mathbb{G}}
\newcommand{\N}{\mathbb{N}}
\newcommand{\R}{\mathbb{R}_{+}^{n}}
\newcommand{\RR}{\mathbb{R}^{n}}

\def\eps{\varepsilon}

\def\t{\theta}
\def\1{d \Omega_{\lambda+\mathbf{1}+\eps}(s)}
\def\q{\mathfrak{q}}

\DeclareMathOperator{\domain}{Dom}
\DeclareMathOperator{\supp}{supp}

\begin{document}

\title[Fundamental operators in Dunkl and Bessel settings]{On fundamental harmonic analysis operators in certain Dunkl and Bessel settings}

\author[A.J. Castro]{Alejandro J. Castro}
\address{Alejandro J. Castro,     \newline
Departamento de An\'alisis Matem\'atico, Universidad de la Laguna       \newline
Campus de Anchieta, Avda. Astrof\'isico Francisco S\'anchez, s/n,       \newline
38271 La Laguna (Sta. Cruz de Tenerife), Spain
      }
\email{ajcastro@ull.es}

\author[T.Z. Szarek]{Tomasz Z. Szarek}
\address{Tomasz Z. Szarek,     \newline
            Instytut Matematyczny,
      Polska Akademia Nauk, \newline
      \'Sniadeckich 8,
      00--956 Warszawa, Poland
      }
\email{szarektomaszz@gmail.com}

\subjclass[2010]{42C05 (primary), 42B20, 42B25 (secondary)}
\keywords{Dunkl Laplacian, Bessel Laplacian, maximal operator, square function, multiplier, Riesz transform, Lusin area integral, Calder\'on-Zygmund operator}
\thanks{The first author was partially supported by MTM2010/17974 and also by an FPU grant from the Government of Spain. 
The second author was partially supported by NCN research project 2012/05/N/ST1/02746}

\begin{abstract}
    We consider several harmonic analysis operators in the multi-dimensional context of the
    Dunkl Laplacian with the underlying group of reflections isomorphic to $\mathbb{Z}_2^n$ (also negative values of the multiplicity function are admitted).
Our investigations include maximal operators, $g$-functions,
    Lusin area integrals, Riesz transforms and multipliers of Laplace and Laplace-Stieltjes transform type.
    Using the general Calderón-Zygmund theory we prove that these objects are bounded in weighted $L^p$ spaces, $1<p<\infty$,
    and from $L^1$ into weak $L^{1}$.
\end{abstract}

\maketitle

\section{Introduction}
In \cite{CaSz} the authors considered various harmonic analysis operators in the Bessel context associated with the (modified) Hankel transform. The present paper is a continuation and extension of that research. We investigate several harmonic analysis operators 
such as heat and Poisson semigroups maximal operators, Littlewood-Paley-Stein mixed $g$-functions, mixed Lusin area integrals, higher order Riesz transforms, multipliers of Laplace and Laplace-Stieltjes transform type (noteworthy these multipliers cover, as special cases, imaginary powers of the Dunkl Laplacian)
in a more general Dunkl setting with the underlying group of reflections isomorphic to $\mathbb{Z}_2^n$. In fact, after restricting to reflection invariant functions this Dunkl situation reduces to the one from \cite{CaSz}.
However, some objects of our interest are defined in a slightly different way than in \cite{CaSz}. This pertains to higher order mixed square functions and Riesz transforms; see the comment following Proposition~\ref{pro:kerestLAIP}.
Moreover, in comparison with \cite{CaSz}, we investigate also mixed Lusin area integrals.
Consequently, our present research delivers also new results in the Bessel setting.

For basic facts concerning the Dunkl framework we refer the reader to the survey article by Rösler \cite{Ros}. Here, we invoke only the most relevant definitions,
which will be needed for our purposes and, in particular, are
connected with the special case when the group of reflections is isomorphic to $\mathbb{Z}_2^n$.

We will work in the space $\mathbb{R}^n$, $n\ge 1$, equipped with the doubling measure
$$
dw_\lambda(x) = \prod_{j=1}^n |x_j|^{2\lambda_j}\, dx,
\qquad x=(x_1,\ldots,x_n).
$$
The multi-index $\lambda = (\lambda_1,\ldots, \lambda_n)$ represents the multiplicity function and will always be assumed to belong to $(-1/2,\infty)^n$.
Notice that negative values of the multiplicity function are admitted.
We consider the group of reflections $G$ generated by $\sigma_j$, $j=1,\ldots,n$,
$$
\sigma_{j}(x_{1},\ldots,x_{j},\ldots,x_{n})=(x_{1},\ldots,-x_{j},\ldots,x_{n}).
$$
Clearly, the reflection $\sigma_j$ is in the hyperplane orthogonal to the $j$th coordinate vector.
The associated differential-difference operators
$$
T_{j}^{\lambda}f(x)
    = \partial_{x_{j}}f(x) + \lambda_{j} \frac{f(x)-f(\sigma_{j}x)}{x_{j}}, \qquad f\in C^{1}(\mathbb{R}^{n}), \quad j=1,\ldots,n,
$$
form a commuting system. The Dunkl Laplacian is defined in a natural way as
$$
\Delta_{\lambda}f
    =-\sum_{j=1}^{n}\big(T_{j}^{\lambda}\big)^{2}f, \qquad f \in C^2(\RR).
$$
This operator will play in our context a similar role to that of the Euclidean Laplacian in the classical harmonic analysis.
Obviously, the trivial choice of the multiplicity function ($\lambda \equiv 0$) reduces our situation to the analysis related to the classical Laplacian.

The study of harmonic analysis operators in the Dunkl setting has
been carried out by many authors in recent years. In particular,
in a general Dunkl context the unweighted $L^p$ mapping properties
for the heat and Poisson semigroups maximal operators were
implicitly established by Thangavelu and Xu in \cite{TX1}.
Recently the latter objects were studied also by Li and Liao \cite{LL} in the one-dimensional situation.
Riesz transforms in the Dunkl setting have also drawn considerable attention. 
$L^p$ mapping properties of the first
order Riesz transforms were investigated in the one-dimensional
situation by Thangavelu and Xu \cite{TX2} and then by Amri, Gasmi
and Sifi in \cite{AGS}. Later on Amri and Sifi \cite{AS} developed
a variant of Calderón-Zygmund theory, which turned out to be well
suited to the general Dunkl framework and, in particular, allowed
them to obtain unweighted $L^p$ bounds for the first order Riesz
transforms. All the above mentioned papers, however, contain a
constraint on the multiplicity function, namely nonnegativity is
required. Here, investigating the case when $G \simeq
\mathbb{Z}_2^n$, we are not so restrictive and allow the multiplicity
function to be negative. Our considerations fit into a recent line
of research connected with analysis for ``low'' values of type
parameters, which was developed in the Bessel
\cite{CaSz},  Jacobi \cite{NSS} and Laguerre \cite{NoSz} settings.
Furthermore, in comparison with \cite{AGS,AS,TX1,TX2} we obtain
weighted $L^p$ estimates with a large class of weights admitted.

This research is also motivated by results obtained recently in a discrete context of the Dunkl harmonic oscillator associated with the Dunkl Laplacian $\Delta_{\lambda}$.
This concerns especially results contained in the papers \cite{NoSt6,NoSt4,Sz1,Sz2}, where weighted $L^p$ mapping properties were studied for several operators such as Riesz transforms, imaginary powers of the Dunkl harmonic oscillator, $g$-functions and Lusin area integrals, and multipliers of Laplace and Laplace-Stieltjes transform type, respectively.
It is worth pointing out that more recently in \cite{A} some unweighted $L^p$ results for Riesz transforms in a general Dunkl harmonic oscillator context were proved.

The main objective of our paper is to analyze $L^p$ mapping properties of various harmonic analysis operators related to the Dunkl Laplacian setting. The main result of the paper, Theorem~\ref{thm:main} below, says that the heat semigroup maximal operator, Littlewood-Paley-Stein mixed $g$-functions, mixed Lusin area integrals, higher order Riesz transforms and multipliers of Laplace and Laplace-Stieltjes transform type are bounded on weighted $L^p$, $1<p<\infty$, spaces and are of weighted weak type $(1,1)$ for a large class of weights.
To prove this we exploit the method from \cite{NoSt4}, see also \cite{Sz1}, which allows us to reduce the analysis to the smaller space $(\R,dw_{\lambda}^+)$
(here and later on $dw_{\lambda}^+$ is the restriction of $dw_\lambda$ to 
$\R \equiv (0,\infty)^n$)
and appropriately defined Bessel-type operators emerging in a natural way from the original ones.
Next, see Theorem~\ref{thm:CZ} below, we show that these auxiliary operators can be interpreted as (vector-valued) Calderón-Zygmund operators associated with the space of homogeneous type $(\R,dw_\lambda^+,|\cdot|)$.
The main technical difficulty connected with this approach is to prove the relevant standard estimates for the kernels involved. The technique we use has its roots in previous papers, see \cite{CaSz,NoSz} and references given there.
As a consequence, by means of standard arguments, we obtain also similar results for analogous operators based on the Poisson semigroup.
This, however, is not so straightforward in the case of Lusin area integrals, where more delicate analysis is needed; see the proof of Proposition~\ref{pro:kerestLAIP}.
The proof of this result gives also an intuition how to deduce a similar result in the Dunkl harmonic oscillator context, see the comment following the statement of \cite[Theorem~2.7]{Sz1}.

The paper is organized as follows. Section~\ref{sec:mainres} contains the setup, definitions of all investigated objects in the Dunkl setting and statements of the main results. We define Bessel-type operators related to the space $(\R,dw_{\lambda}^+,|\cdot|)$ and reduce proving the main theorem to showing that these auxiliary operators are (vector-valued)
Calder\'on-Zygmund operators related to this smaller space.
This section ends with various comments pertaining to the main results.
In Section~\ref{sec:as} we prove that the Bessel-type objects are $L^2(dw_\lambda^+)$-bounded.
Finally, in Section~\ref{sec:kerest} we obtain standard estimates
(see \eqref{gr}-\eqref{sm2} below) for all kernels associated with the Bessel-type operators. This is the most technical part of the paper.

\textbf{Notation.}
Throughout the paper we use a fairly standard notation with essentially all symbols referring to the spaces
of homogeneous type $(\RR,dw_{\lambda},|\cdot|)$ and $(\R,dw^+_{\lambda},|\cdot|)$. Here and later on
$dw^+_{\lambda}$ stands for the restriction of $dw_{\lambda}$ to $\R$.
For the sake of clarity, we now explain all
symbols and relations that might lead to a confusion.
We denote by $C_c^\infty(\R)$ the space of smooth and compactly supported functions in $\R$. By $\langle f,g \rangle_{dw_\lambda^+}$ we mean $\int_{\R} f(x) \overline{g(x)} \, dw_{\lambda}^+(x)$ whenever the integral makes sense. By $L^p(\R,Udw_\lambda^+)$ we understand the weighted $L^p(dw_\lambda^+)$ space, $U$ being a nonnegative weight on $\R$; we write simply $L^p(dw_\lambda^+)$ when $U \equiv 1$. 
Further, for $1 \le p < \infty$ we write $A_p^{\lambda,+}$ for the Muckenhoupt class of $A_p$ weights connected with the space $(\R,dw^+_{\lambda},|\cdot|)$.
Given $x,y \in \R$, $\beta \in \RR$ and
$M \in \mathbb{N}^n$, $\mathbb{N}=\{0,1,2,\ldots\}$, we denote
\begin{align*}
\mathbf{1} & = (1,\ldots,1) \in \mathbb{N}^n,\\
|\lambda| & = \lambda_1 + \ldots + \lambda_n,\\
|x| & = (x_1^2 + \ldots + x_n^2)^{1/2}, \qquad 
\textrm{(Euclidean distance)}\\
B(x,r) & = \{y \in \R : |y-x|<r\}, \qquad r>0, \qquad
    \textrm{(balls in $\R$)}\\
xy & = (x_1 y_1,\ldots, x_n y_n),\\
x \vee y & = (\max\{x_1, y_1\},\ldots, \max\{x_n, y_n\}),\\
x \wedge y & = (\min\{x_1, y_1\},\ldots, \min\{x_n, y_n\}),\\
x^\beta & = x_1^{\beta_1}\cdot \ldots \cdot x_n^{\beta_n},\\
x \le y & \equiv x_j \le y_j, \qquad j=1,\ldots,n,\\
\lfloor x \rfloor & = (\max\{ k \in \mathbb{Z} : k\le x_1\}, \ldots, \max\{ k \in \mathbb{Z} : k\le x_n\}) , \qquad
        \textrm{(floor function)} \\
\overline{M} & = (\overline{M_1}, \ldots, \overline{M_n}), \qquad
\overline{M_j}= M_j-2\lfloor M_j/2 \rfloor = \chi_{ \{M_j \text{ is odd} \}}, \\
(T^{\lambda})^M & = (T_1^{\lambda})^{M_1} \circ \ldots \circ (T_n^{\lambda})^{M_n}.
\end{align*}
In an analogous way we define $C_c^\infty(\RR)$, $\langle f,g \rangle_{dw_\lambda}$ and $L^p(\RR,Wdw_\lambda)$. Furthermore, if $x,y \in \RR$ and $\beta \in \mathbb{N}^n$ we understand the objects $xy$, $x^\beta$, $\lfloor x \rfloor$ and relation $x \le y$ in the same way as above whenever it makes sense.
We shall also
use the following terminology. Given $\eta \in \mathbb{Z}_2^n$, we say that a function $f \colon \RR \to \mathbb{C}$ is $\eta$-symmetric if for each $j=1,\ldots,n$, $f$ is either even or odd with respect to the $j$th coordinate according to whether $\eta_{j}=0$ or $\eta_{j}=1$, respectively. If $f$ is $(0,\ldots,0)$-symmetric, then we simply say that $f$ is symmetric.
Further, if there exists $\eta\in \mathbb{Z}_2^n$ such that $f$ is $\eta$-symmetric, then we denote by $f^{+}$ its restriction to $\R$.
Finally, $f_\eta$ stands for the $\eta$-symmetric component of $f$, namely
\[
f = \sum_{\eta \in \{ 0,1 \}^n} f_\eta, \qquad
f_\eta (x) = \frac{1}{2^n}\sum_{\eps \in \{-1, 1 \}^n} \eps^\eta f(\eps x).
\] 
Conversely, if $f \colon \R \to \mathbb{C}$, then by $f^\eta$ we mean the $\eta$-symmetric extension of $f$ to the space $\RR$, i.e.
\begin{align*}
f^\eta (x)
    = \left\{\begin{array}{ll}
         \eps^\eta f(\eps x), & \text{if $x\in (\mathbb{R}\setminus\{0\})^n$ and $\eps \in \{-1 , 1 \}^n$ is such that $\eps x \in \R$,} \\
         &\\
         0, & \text{if $x \notin (\mathbb{R}\setminus\{0\})^n$.}
      \end{array}\right.
\end{align*}

While writing estimates, we will use the notation $X \lesssim Y$ to
indicate that $X \le CY$ with a positive constant $C$ independent of significant quantities. We shall
write $X \simeq Y$ when simultaneously $X \lesssim Y$ and $Y \lesssim X$.

\section{Preliminaries and main results} \label{sec:mainres}

Let $\lambda \in (-1/2,\infty)^n$. For $z \in \RR$ we consider the functions $\psi_z^\lambda$ and $\varphi_z^\lambda$, which are given as the tensor products
\begin{align}\label{defpsi}
\psi_z^\lambda(x) &= \prod_{j=1}^n \psi_{z_j}^{\lambda_j}(x_j), \qquad
 \varphi_z^\lambda(x) = \prod_{j=1}^n \varphi_{z_j}^{\lambda_j}(x_j), \qquad x \in \RR, \nonumber \\
\psi_{z_j}^{\lambda_j}(x_j) &=  \varphi_{z_j}^{\lambda_j}(x_j) + i x_j z_j \varphi_{z_j}^{\lambda_j+1}(x_j) , \qquad j=1,\ldots,n,
\end{align}
where the function $\varphi_{z_j}^{\lambda_j}(x_j)$ is even with respect to both $x_j$ and $z_j$ and it is given by
\begin{equation*}
\varphi_{z_j}^{\lambda_j}(x_j)
=
\frac{J_{\lambda_j-1/2} (x_jz_j) }{ (x_j z_j)^{\lambda_j-1/2} }
=
\sum_{m=0}^\infty \frac{ (-1)^m (x_j z_j/2)^{2m} }{ 2^{\lambda_j-1/2} \, m! \,\Gamma(m+\lambda_j+1/2) }.
\end{equation*}
Here $J_\nu$ denotes the Bessel function of the first kind and order $\nu$, cf. \cite{Wat}. It is known that for each $z \in \RR$,
the function $\psi_z^\lambda$ is an eigenfunction of the Dunkl Laplacian $\Delta_{\lambda}$ with the corresponding eigenvalue
$|z|^2 = z_1^2 + \ldots + z_n^2$. More precisely,
\begin{equation}\label{Dunkleigenfv}
\Delta_{\lambda} \psi_z^\lambda = |z|^2 \psi_z^\lambda, \qquad z\in \RR.
\end{equation}
The Dunkl transform connected with the Dunkl setting associated with 
$G \simeq \mathbb{Z}_2^n$ is defined for
sufficiently regular functions, say $f\in C_c^{\infty}(\RR)$, by
\[
D_{\lambda} f(z) = \frac{1}{2^n}\int_{\RR}  \overline{ \psi_z^\lambda(x) } f(x)\, dw_{\lambda}(x), \qquad z \in \RR.
\]
It is known that the Dunkl transform is an isometry in $L^2(dw_\lambda)$
and its inverse $\check{D}_\lambda$ is given by $\check{D}_\lambda f(z) = D_\lambda f(-z)$.
For $\lambda \in [0,\infty)^n$ this follows from the general Dunkl theory, see \cite[Theorem~4.26]{deJ}, and for $\lambda \in (-1/2,\infty)^n$ it can easily be deduced from the one-dimensional result \cite[Proposition 1.3]{NoSt3}.
Note that for $\lambda = 0$ the Dunkl transform is just the classical Fourier transform.

We consider the nonnegative self-adjoint extension of $\Delta_{\lambda}$, which will be still denoted by $\Delta_{\lambda}$, defined by
\[
\Delta_{\lambda} f = \check{D}_\lambda(  |z|^2 D_\lambda f(z)  )
\]
on the domain
\[
\domain(\Delta_{\lambda})
=
\{ f \in L^2(dw_\lambda) : |z|^2 D_\lambda f(z) \in L^2(dw_\lambda) \}.
\]
The heat-Dunkl semigroup $\{ \mathbb{W}_t^\lambda \}_{t>0} = \{ e^{-t\Delta_{\lambda}} \}_{t>0}$ generated by $- \Delta_{\lambda}$ is given on $L^2(dw_\lambda)$ by
\begin{equation*}
\mathbb{W}_t^\lambda f
=
\check{D}_\lambda \big(  e^{-t|z|^2} D_\lambda f(z)  \big).
\end{equation*}
It has the integral representation
\[
\mathbb{W}_t^\lambda f(x) = \int_{\RR} \G_t^\lambda (x,y) f(y) \, dw_\lambda(y), \qquad x \in \RR, \quad t>0,
\]
with the kernel
\begin{equation}\label{kerG}
\G_t^\lambda (x,y) =
\frac{1}{2^n}
\sum_{\eta \in \{ 0, 1 \}^n }
(xy)^\eta W_t^{\lambda + \eta}(x,y), \qquad x,y \in \RR, \quad t>0,
\end{equation}
where
\begin{equation*}
W_t^{\lambda}(x,y) =
\frac{1}{(2t)^n} \exp\Big( -\frac{1}{4t}\big(|x|^2+|y|^2\big)\Big)
    \prod_{j=1}^n (x_j y_j)^{-\lambda_j+1\slash 2} I_{\lambda_{j}-1/ 2}
        \Big(\frac{x_j y_j}{2t} \Big), \quad  x,y \in \RR, \ t>0.
\end{equation*}
Here $I_\nu$ denotes the modified Bessel function of the first kind and order $\nu > -1$, cf. \cite[p.\,395]{Wat} or \cite[(1.1)]{CaSz}. Moreover,
\[
\frac{I_\nu(z)}{z^\nu} =
\sum_{m=0}^\infty \frac{ (z/2)^{2m} }{ 2^\nu  m! \, \Gamma(m + \nu + 1) },
\qquad z \in \mathbb{R}, \quad \nu>-1.
\]
Observe that $W_t^{\lambda}(x,y)$ restricted to $(x,y) \in \R \times \R$ is the heat-Bessel kernel considered for instance in \cite{BCN,CaSz}.

Now we are ready to introduce the main objects of our study, which are defined initially in $L^2(dw_\lambda)$ in the cases (1)-(4) and (6), or in $C^{\lambda}$
(the space of smooth $L^2(dw_\lambda)$-functions whose
Dunkl transform is also smooth and compactly supported in $\RR\setminus \{ 0 \}$) in the case of Riesz transforms (5).

\begin{itemize}
\item[$(1)$] The heat-Dunkl semigroup maximal operator
$$
\mathbb{W}_{*}^{\lambda}f = \big\| \mathbb{W}_t^{\lambda}f\big\|_{L^{\infty}(dt)}.
$$
\item[$(2)$] Littlewood-Paley-Stein type mixed $g$-functions
$$
g^{\lambda}_{K,M}(f) = \big\| \partial_t^K (T^\lambda)^M \mathbb{W}_t^{\lambda}f \big\|_{L^2(t^{2K+|M|-1}dt)},
$$
where $M \in \N^n$, $K \in \N$, $|M|+K>0$.
\item[$(3)$] Multipliers of Laplace transform type
$$
\mathcal{M}^{\lambda}_{\mathfrak{m}} f = \check{D}_{\lambda}(\mathfrak{m}D_{\lambda}f),
$$
where $\mathfrak{m}(z) = |z|^2 \int_0^{\infty} e^{-t|z|^2} \Phi(t)\, dt$ with $\Phi \in L^{\infty}(dt)$.
\item[$(4)$] Multipliers of Laplace-Stieltjes transform type
$$
\mathcal{M}^{\lambda}_{\mathfrak{m}} f = \check{D}_{\lambda}(\mathfrak{m}D_{\lambda}f),
$$
where $\mathfrak{m}(z) = \int_{(0,\infty)} e^{-t|z|^2} \, d\nu (t)$, with
$\nu$ being a complex Borel measure on $(0,\infty)$.
\item[$(5)$] Riesz transforms of order $M$
$$
R_M^{\lambda}f = (T^\lambda)^M \check{D}_{\lambda} (|z|^{-|M|} D_{\lambda} f(z) ),
$$
where $M \in \N^n$ and $|M|>0$.
Note that similar arguments to those used in the proof of Proposition~\ref{pro:L2b} (the case of $R_M^{\lambda,\eta,+}$), see Section~\ref{sec:as} below, show that $C^{\lambda}$ is dense in $L^2(dw_\lambda)$.
\item[$(6)$] Mixed Lusin area type integrals
$$
S^{\lambda}_{K,M}(f)(x)=
\bigg( \int_{A(x)} t^{2K+|M|-1} \big|\partial_{t}^K (T^\lambda)^M \mathbb{W}_{t}^{\lambda}f(z)\big|^{2}\frac{dw_{\lambda}(z) \, dt}
{V_{\sqrt{t}}^{\lambda}(x)} \bigg)^{1\slash 2},
$$
where $M \in \N^n$, $K \in \N$, $|M|+K>0$, $A(x)$ is the parabolic cone with vertex at $x$,
\begin{equation}\label{def:A}
A(x)=(x,0)+A,\qquad
A=\Big\{(z,t)\in \RR \times(0,\infty) : |z|<\sqrt{t}\Big\},
\end{equation}
and $V_{t}^{\lambda}(x)$ is the $w_{\lambda}$ measure of the cube centered at $x$ and of side lengths $2t$. More precisely,
\begin{align*}
V_{t}^{\lambda}(x)=\prod_{j=1}^{n}V_{t}^{\lambda_{j}}(x_j),\qquad
V_{t}^{\lambda_j}(x_j)=w_{\lambda_{j}}\big((x_{j}-t,x_{j}+t)\big),\qquad x\in\RR,\quad t>0.
\end{align*}
\end{itemize}

Our main result reads as follows.

\begin{thm}\label{thm:main}
Assume that $\lambda \in(-1\slash 2,\infty)^{n}$ and $W$ is a weight on $\RR$ invariant under the reflections $\sigma_1, \ldots ,\sigma_n$.
Let $W^+$ be the restriction of $W$ to $\R$. Then the multipliers of Laplace and Laplace-Stieltjes transform type and the Riesz transforms extend to
bounded linear operators on $L^{p}(\RR,Wdw_{\lambda})$, $W^{+}\in A_{p}^{\lambda,+}$, $1<p<\infty$, and from $L^{1}(\RR,Wdw_{\lambda})$ to weak $L^{1}(\RR,Wdw_{\lambda})$, $W^{+}\in A_{1}^{\lambda,+}$.
Furthermore, the heat-Dunkl semigroup maximal operator, the mixed $g$-functions and the mixed Lusin area integrals are bounded on $L^{p}(\RR,Wdw_{\lambda})$, $W^{+}\in A_{p}^{\lambda,+}$, $1<p<\infty$, and from $L^{1}(\RR,Wdw_{\lambda})$ to weak $L^{1}(\RR,Wdw_{\lambda})$, $W^{+}\in A_{1}^{\lambda,+}$.
\end{thm}

Notice that for symmetric functions the condition $W^{+}\in A_{p}^{\lambda,+}$ is equivalent to saying that $W$ is in the Muckenhoupt class of $A_p$ weights associated with the initial space $(\RR,dw_\lambda,|\cdot|)$.

The proof of Theorem \ref{thm:main} can be reduced to showing analogous properties for certain, suitably defined, auxiliary operators emerging
from those introduced above and related to the smaller space $(\R,dw_\lambda^+,|\cdot|)$.
To proceed, for each $\eta \in \{ 0,1 \}^n$ we consider an auxiliary semigroup acting initially on $L^2(dw_\lambda^+)$ and given by the formula
\begin{equation*}
\mathbb{W}_t^{\lambda,\eta,+} f
=
\Big( \check{D}_\lambda \big( e^{-t|z|^2} D_\lambda f^\eta (z) \big) \Big)^+.
\end{equation*}
These semigroups have the integral representations, see \eqref{kerG},
\begin{align*}
\mathbb{W}_t^{\lambda,\eta,+} f (x) & =
\int_{\R} \G_t^{\lambda,\eta,+} (x,y) f(y) dw_\lambda^+(y),
\qquad x\in \R, \quad t>0,\\
\G_t^{\lambda,\eta,+} (x,y) & =
(xy)^\eta W_t^{\lambda + \eta} (x,y), \qquad x,y \in \R, \quad t>0.
\end{align*}
Further, these integral formulas provide us a good definition of $\mathbb{W}_t^{\lambda,\eta,+}$ on weighted $L^p$ spaces for a large class of weights and produce always smooth functions of $(x,t) \in \R \times \mathbb{R}_+$, see Lemma~\ref{lem:3.5BCN} below for more details.

For $\eta \in \{ 0,1 \}^n$ and $M \in \mathbb{N}^n$ we denote
$\delta_{\eta,M}  = \delta_{1, \eta_1,M_1} \circ \ldots \circ \delta_{n, \eta_n,M_n}$, being for every $j=1, \ldots, n,$
\begin{align*}
\delta_{j, 0,M_j}
    = \left\{\begin{array}{ll}
         (\delta_j^* \delta_j)^{M_j/2}, & \text{if } M_j \text{ is even,} \\
         &\\
         \delta_j (\delta_j^* \delta_j)^{( M_j - 1)/2}, & \text{if } M_j \text{ is odd,}
      \end{array}\right. \quad
\delta_{j, 1,M_j}
    = \left\{\begin{array}{ll}
         (\delta_j \delta_j^*)^{M_j/2}, & \text{if } M_j \text{ is even,} \\
         &\\
         \delta_j^* (\delta_j \delta_j^*)^{( M_j - 1) /2}, & \text{if } M_j \text{ is odd,}
      \end{array}\right.
\end{align*}
where $\delta_j  = \partial_{x_j}$ and $\delta_j^*  = \partial_{x_j} + \frac{2\lambda_j}{x_j}$ is, up to a sign, the formal adjoint of $\delta_j$ in the space $L^2(dw_\lambda^+)$.
These derivatives correspond to the action of $(T^\lambda)^M$ on $\eta$-symmetric functions. To be more precise, if $f$ is $\eta$-symmetric, then $(T^\lambda)^M f = \delta_{\eta,M} f$.
Moreover, we may also think that each $\delta_{\eta,M}$ acts on functions defined on the restricted space $\R$.

Now we are ready to introduce the auxiliary Bessel-type operators, which are defined initially in $L^2(dw_\lambda^+)$ in the cases (1)-(4) and (6), or in
$$
C^{\lambda,\eta,+} = \{ f \in L^2(dw_\lambda^+) :
f^{\eta} \in C^\infty(\RR),
\big( D_{\lambda} f^{\eta} \big)^+ \in C_c^\infty(\R) \}
$$
in the case of the Bessel-type Riesz transforms (5).
\begin{itemize}
\item[$(1)$] The maximal operator
$$
\mathbb{W}_{*}^{\lambda,\eta,+}f = \big\| \mathbb{W}_t^{\lambda,\eta,+}f\big\|_{L^{\infty}(dt)}.
$$
\item[$(2)$] Littlewood-Paley-Stein type mixed $g$-functions
$$
g^{\lambda,\eta,+}_{K,M}(f) = \big\| \partial_t^K \delta_{\eta,M} \mathbb{W}_t^{\lambda,\eta,+}f \big\|_{L^2(t^{2K+|M|-1}dt)},
$$
where $M \in \N^n$, $K \in \N$, $|M|+K>0$.
\item[$(3)$] Multipliers of Laplace transform type
$$
\mathcal{M}^{\lambda,\eta,+}_{\mathfrak{m}} f =
\big( \check{D}_{\lambda}(\mathfrak{m}D_{\lambda}f^{\eta}) \big)^+,
$$
where $\mathfrak{m}(z) = |z|^2 \int_0^{\infty} e^{-t|z|^2} \Phi(t)\, dt$ with $\Phi \in L^{\infty}(dt)$.
\item[$(4)$] Multipliers of Laplace-Stieltjes transform type
$$
\mathcal{M}^{\lambda,\eta,+}_{\mathfrak{m}} f =
\big( \check{D}_{\lambda}(\mathfrak{m}D_{\lambda}f^{\eta}) \big)^+,
$$
where $\mathfrak{m}(z) = \int_{(0,\infty)} e^{-t|z|^2} \, d\nu (t)$, with
$\nu$ being a complex Borel measure on $(0,\infty)$.
\item[$(5)$] Riesz transforms of order $M$
$$
R_M^{\lambda,\eta,+}f = \delta_{\eta,M} \big( \check{D}_{\lambda} (|z|^{-|M|} D_{\lambda} f^{\eta} (z) ) \big)^+,
$$
where $M \in \N^n$ and $|M|>0$.
\item[$(6)$] Mixed Lusin area type integrals
$$
S^{\lambda,\eta,+}_{K,M}(f)(x)=
\bigg( \int_{A(x)} t^{2K+|M|-1} \big|\partial_{t}^K \delta_{\eta,M} \mathbb{W}_{t}^{\lambda,\eta,+}f(z)\big|^{2}
\, \chi_{ \{z \in \R \} }
\frac{dw^+_{\lambda}(z) \, dt}
{V_{\sqrt{t}}^{\lambda,+}(x)} \bigg)^{1\slash 2},
$$
where $M \in \N^n$, $K \in \N$, $|M|+K>0$, and $A(x)$ is the parabolic cone with vertex at $x$, see \eqref{def:A}.
Here $V_{t}^{\lambda,+}(x)$ is the $w^+_{\lambda}$ measure of the cube centered at $x$ and of side lengths $2t$, restricted to $\R$. More precisely,
\begin{align*}
V_{t}^{\lambda,+}(x)=\prod_{j=1}^{n}V_{t}^{\lambda_{j},+}(x_j),\qquad x\in\R,\quad t>0,
\end{align*}
and for $j=1,\ldots,n$,
\begin{align} \nonumber
V_{t}^{\lambda_j,+}(x_j)
    &= w_{\lambda_{j}}^{+}\big((x_{j}-t,x_{j}+t)\cap\mathbb{R}_{+}\big) \\ \label{Vexplicit}
    &= \left( (x_{j}+t)^{2\lambda_j+1} - \chi_{ \{ x_j > t \}} (x_{j}-t)^{2\lambda_j+1} \right)/(2\lambda_j+1).
\end{align}
\end{itemize}
Notice that the Bessel-type Lusin area integrals can be written as
\begin{equation*}
S^{\lambda,\eta,+}_{K,M}(f)(x)=
\big\|\partial_{t}^K \delta_{\eta,M} \mathbb{W}_{t}^{\lambda,\eta,+}f(x+z)\sqrt{\Xi_{\lambda}(x,z,t)}
\,\chi_{\{x+z\in\R\}}\big\|_{L^2(A,t^{ 2K + |M| - 1}dzdt)},
\end{equation*}
where the function $\Xi_{\lambda}$ is given by
\begin{equation*}
\Xi_{\lambda}(x,z,t)=\prod_{j=1}^{n}
\frac{(x_{j}+z_{j})^{2\lambda_{j}}}{V_{\sqrt{t}}^{\lambda_{j},+}(x_{j})},\qquad x\in\R,\quad z\in\RR,\quad x+z\in\R.
\end{equation*}

Similar arguments to those given in \cite[p.\,6]{NoSt4} and \cite[pp.\,1522--1524]{Sz1} allow us to reduce the proof of Theorem~\ref{thm:main} by showing the following.

\begin{thm}\label{thm:main+}
Assume that $\lambda \in(-1\slash 2,\infty)^{n}$ and $\eta \in \{ 0, 1\}^n$. Then the Bessel-type operators \emph{(3)}-\emph{(5)} extend to
bounded linear operators on $L^{p}(\R,Udw^+_{\lambda})$, $U \in A_{p}^{\lambda,+}$, $1<p<\infty$, and from $L^{1}(\R,Udw^+_{\lambda})$ to weak $L^{1}(\R,Udw^+_{\lambda})$, $U \in A_{1}^{\lambda,+}$.
Furthermore, the sublinear operators \emph{(1)}, \emph{(2)} and \emph{(6)} are bounded on $L^{p}(\R,Udw^+_{\lambda})$, $U \in A_{p}^{\lambda,+}$, $1<p<\infty$, and from $L^{1}(\R,Udw^+_{\lambda})$ to weak $L^{1}(\R,Udw^+_{\lambda})$, $U \in A_{1}^{\lambda,+}$.
\end{thm}

To prove Theorem \ref{thm:main+} we will use the general (vector-valued) Calder\'on-Zygmund theory.
In fact we are going to show that the Bessel-type operators (1)-(6) are (vector-valued) Calder\'on-Zygmund
operators in the sense of the space of homogeneous type $(\R,dw_\lambda^+,|\cdot|)$. Then, in particular,
the mapping properties claimed in Theorem \ref{thm:main+} will follow from the general theory and arguments
similar to those mentioned in \cite[Section 2]{BCN}. To proceed we shall need a slightly more general definition
of the standard kernel, or rather standard estimates, than the one used in the papers \cite{BCN,CaSz}.
More precisely, we will allow slightly weaker smoothness estimates as indicated below, see for instance \cite{Sz1}.

Let $\mathbb{B}$ be a Banach space and let $K(x,y)$ be a kernel defined on
$\R\times\R\backslash \{(x,y):x=y\}$ and taking values in $\mathbb{B}$.
We say that $K(x,y)$ is a standard kernel in the sense of the space of homogeneous type
$(\R, dw^+_{\lambda},|\cdot|)$ if it satisfies
the growth estimate
\begin{equation} \label{gr}
\|K(x,y)\|_{\mathbb{B}} \lesssim \frac{1}{w^+_{\lambda}(B(x,|x-y|))}
\end{equation}
and the smoothness estimates
\begin{align}
\| K(x,y)-K(x',y)\|_{\mathbb{B}} & \lesssim \bigg(\frac{|x-x'|}{|x-y|} \bigg)^{\gamma}\, \frac{1}{w^+_{\lambda}(B(x,|x-y|))},
\qquad |x-y|>2|x-x'|, \label{sm1}\\
\| K(x,y)-K(x,y')\|_{\mathbb{B}} & \lesssim \bigg(\frac{|y-y'|}{|x-y|} \bigg)^{\gamma}\, \frac{1}{w^+_{\lambda}(B(x,|x-y|))},
\qquad |x-y|>2|y-y'| \label{sm2},
\end{align}
for some fixed $\gamma > 0$. Notice that the bounds \eqref{sm1} and \eqref{sm2} imply analogous estimates with any $0<\gamma'<\gamma$ instead of $\gamma$.
Further, observe that in these formulas, the ball $B(x,|y-x|)$ can be replaced by $B(y,|x-y|)$, in view of
the doubling property of $w^+_{\lambda}$.
Furthermore, when $K(x,y)$ is scalar-valued (i.e.\ $\mathbb{B}=\mathbb{C}$) and $\gamma = 1$, the difference bounds \eqref{sm1}
and \eqref{sm2} are implied by the more convenient gradient estimate
\begin{equation} \label{grad}
|\nabla_{\! x,y} K(x,y)| \lesssim \frac{1}{|x-y|w^+_{\lambda}(B(x,|x-y|))}.
\end{equation}
Similar reduction holds also in the vector-valued situations we consider; see the comments and arguments presented in \cite[Section 4]{NSS}.
Here, however, we will also use \eqref{sm1} and \eqref{sm2} with $\gamma <1$ and thus it is convenient to analyze the smoothness estimates rather than \eqref{grad}.

A linear operator $T$ assigning to each $f\in L^2(dw^+_{\lambda})$ a measurable $\B$-valued function $Tf$ on $\R$ is a (vector-valued) Calder\'on-Zygmund operator in the sense of the space $(\R,dw^+_{\lambda},|\cdot|)$ if
\begin{itemize}
    \item[$(i)$] $T$ is bounded from $L^2(dw^+_{\lambda})$ to $L^2_{\B}(dw^+_{\lambda})$,
    \item[$(ii)$] there exists a standard $\B$-valued kernel $K(x,y)$ such that
\begin{align*}
Tf(x)=\int_{\R}K(x,y)f(y)\,dw^+_{\lambda}(y),\qquad \textrm{a.a.}\,\,\, x\notin \supp f,
\end{align*}
for every $f \in L_c^{\infty}(dw^+_{\lambda})$,
where $L_c^{\infty}(dw^+_{\lambda})$ is the subspace of $L^{\infty}(dw^+_{\lambda})$ of bounded measurable functions
with compact supports.
\end{itemize}
Here integration of $\mathbb{B}$-valued functions is understood in Bochner's sense, and
$L^2_{\mathbb{B}}(dw^+_{\lambda})$ is the Bochner-Lebesgue space of all $\mathbb{B}$-valued
$dw^+_{\lambda}$-square integrable functions on $\R$.

It is well known that a large part of the classical theory of Calder\'on-Zygmund operators remains valid,
with appropriate adjustments, when the underlying space is of homogeneous type and the associated kernels
are vector-valued, see for instance \cite{RRT} and \cite{RT}.

The following result together with the arguments discussed above imply Theorem~\ref{thm:main+} and thus also Theorem \ref{thm:main}.

\begin{thm}\label{thm:CZ}
Assume that $\lambda \in (-1\slash 2,\infty)^{n}$ and $\eta\in \{ 0 , 1 \}^n$. Then the Bessel-type operators
\begin{align*}
\mathbb{W}_{*}^{\lambda,\eta,+}, \quad
g^{\lambda,\eta,+}_{K,M}, \quad
\mathcal{M}^{\lambda,\eta,+}_{\mathfrak{m}}, \quad
\mathcal{M}^{\lambda,\eta,+}_{\mathfrak{m}}, \quad
R_M^{\lambda,\eta,+}, \quad
S^{\lambda,\eta,+}_{K,M},
\end{align*}
are (vector-valued) Calder\'on-Zygmund operators in the sense of the space of homogeneous type $(\R,dw_{\lambda}^+,|\cdot|)$ associated with the Banach spaces $\mathbb{B}$, where $\mathbb{B}$ is $C_0$, $L^2(t^{2K+|M|-1}dt)$, $\mathbb{C}$, $\mathbb{C}$, $\mathbb{C}$, $L^2(A,t^{2K+|M|-1}dzdt)$, respectively.
Here $C_0$ denotes the closed separable subspace of $L^{\infty}(dt)$ consisting of all continuous functions on
$\mathbb{R}_{+}$ which have finite limits as $t\to 0^{+}$ and vanish as $t\to \infty$.
\end{thm}

Proving Theorem \ref{thm:CZ} splits naturally into showing the following three results.

\begin{pro}\label{pro:L2b}
Let $\lambda\in(-1\slash 2,\infty)^{n}$ and $\eta\in \{ 0 , 1 \}^n$. Then the Bessel-type operators from Theorem \ref{thm:CZ} are bounded on $L^2(dw_{\lambda}^{+})$.
\end{pro}

Formal computations and the results from papers \cite{BCN,CaSz} suggest that the Bessel-type operators are associated with the following kernels related to appropriate Banach spaces $\mathbb{B}$.

\begin{itemize}
\item[$(1)$] The kernel associated with Bessel-type maximal operator $\mathbb{W}_*^{\lambda, \eta, +}$ 
$$
\mathcal{W}^{\lambda, \eta, +}(x,y) = \big\{\G_t^{\lambda,\eta,+}(x,y)\big\}_{t>0}, \qquad \mathbb{B} = C_0 \subset L^{\infty}(dt).
$$
Using formula \eqref{Bhk} below
it can easily be justified that $\mathcal{W}^{\lambda, \eta, +}(x,y) \in C_0$ for $x \ne y$.
\item[$(2)$] The kernels associated with Bessel-type mixed $g$-functions $g^{\lambda,\eta,+}_{K,M}$ 
$$
\mathcal{G}_{K,M}^{\lambda,\eta,+} (x,y) = \big\{ \partial_t^K \delta_{\eta,M,x} \G_t^{\lambda,\eta,+}(x,y) \big\}_{t>0}, \qquad
    \mathbb{B} = L^2(t^{2K+|M|-1}dt),
$$
where $M \in \N^n$ and $K \in \N$ are such that $|M|+K>0$.
\item[$(3)$] The kernels associated with Laplace transform type multipliers $\mathcal{M}^{\lambda,\eta,+}_{\mathfrak{m}}$ 
$$
K^{\lambda,\eta,+}_{\Phi}(x,y) = - \int_0^{\infty} \Phi(t) \partial_t \G_t^{\lambda,\eta,+}(x,y)\, dt, \qquad
    \mathbb{B}=\mathbb{C},
$$
where $\Phi \in L^{\infty}(dt)$.
\item[$(4)$] The kernels associated with Laplace-Stieltjes transform type multipliers $\mathcal{M}^{\lambda,\eta,+}_{\mathfrak{m}}$ 
$$
K^{\lambda,\eta,+}_{\nu}(x,y) =  \int_{(0,\infty)} \G_t^{\lambda,\eta,+}(x,y)\, d\nu(t), \qquad
    \mathbb{B}=\mathbb{C},
$$
where $\nu$ is a complex Borel measure on $(0,\infty)$.
\item[$(5)$] The kernels associated with Bessel-type Riesz transforms $R_M^{\lambda,\eta,+}$ 
$$
R_M^{\lambda,\eta,+}(x,y) = \frac{1}{\Gamma(|M|\slash 2)} \int_0^{\infty} \delta_{\eta,M,x} \G_t^{\lambda,\eta,+}(x,y)
    t^{|M|\slash 2 -1}\, dt, \qquad \mathbb{B}=\mathbb{C},
$$
where $M \in \N^n$ is such that $|M| > 0$.
\item[$(6)$] The kernels associated with Bessel-type mixed Lusin area integrals $S^{\lambda,\eta,+}_{K,M}$ 
\[
\mathcal{S}^{\lambda,\eta,+}_{K,M}(x,y)
=
\left\{
\partial_{t}^K \delta_{\eta,M,\mathbf{x}} \G_{t}^{\lambda,\eta,+}(\mathbf{x},y)
\Big|_{\mathbf{x} = x+z} \sqrt{\Xi_{\lambda}(x,z,t)}
\,\chi_{\{x+z\in\R\}} \right\}_{(z,t) \in A}
\]
with $\mathbb{B} = L^2(A,t^{2K+|M|-1}dzdt)$,
where $M \in \N^n$ and $K \in \N$ are such that $|M|+K>0$.
\end{itemize}

The following result shows that this is true in the Calder\'on-Zygmund theory sense.

\begin{pro}\label{pro:kerassoc}
Assume that $\lambda \in (-1/2, \infty)^n$ and $\eta \in \{ 0 , 1 \}^n$. Then the Bessel-type operators \emph{(1)}-\emph{(6)} are associated, in the Calder\'on-Zygmund theory sense, with the corresponding kernels just listed.
\end{pro}

\begin{proof}
In the cases of
                $\mathbb{W}_{*}^{\lambda,\eta,+}$,
                $g^{\lambda,\eta,+}_{K,M}$,
                $\mathcal{M}^{\lambda,\eta,+}_{\mathfrak{m}}$ and
                $R_M^{\lambda,\eta,+}$
    we can proceed as in \cite[Section 4]{BCN}.
    The crucial facts needed in the reasoning are Lemma~\ref{lem:3.5BCN} and Propositions~\ref{pro:L2b} and \ref{pro:kerest}, which provide $L^2(dw_{\lambda}^+)$-boundedness of the investigated operators and the growth estimates for the corresponding kernels
    (in some places we need slightly stronger estimates than growth condition, which nevertheless are established in Section \ref{sec:kerest}).
    The case of $S^{\lambda,\eta,+}_{K,M}$ can be dealt with in a similar way. To be precise, we can proceed in much the same way as in
\cite[Proposition 2.5, pp.\,1528-1529]{Sz1}, where Lusin area integrals, in the context of discrete Dunkl setting, were investigated.
    We leave details to the reader.
    \end{proof}

\begin{pro}\label{pro:kerest}
Let $\lambda \in (-1/2, \infty)^n$ and $\eta \in \{ 0 , 1 \}^n$. Then the kernels \emph{(1)}-\emph{(6)} listed above satisfy the standard estimates
with the relevant Banach spaces $\mathbb{B}$. More precisely, the kernels \emph{(1)}-\emph{(5)} satisfy the smoothness conditions with $\gamma = 1$ and the
kernel \emph{(6)} satisfies \eqref{sm1} and \eqref{sm2} with any $\gamma \in (0,1/2]$ such that $\gamma < \min_{1 \le k \le n} (\lambda_k + 1/2)$.
\end{pro}
The proof of Proposition~\ref{pro:L2b} is given in Section~\ref{sec:as} and the proof of Proposition~\ref{pro:kerest}, which is the most technical part of the paper, is located in Section~\ref{sec:kerest}.

We conclude this section with various comments and remarks connected with our main results. We first note that our results imply analogous results for similar objects based on Poisson-Dunkl semigroup.
More precisely, let $\{\mathbb{P}_t^{\lambda}\}_{t>0}$ be the Poisson-Dunkl semigroup generated by $-\sqrt{\Delta_{\lambda}}$,
\[
\mathbb{P}_t^{\lambda} f =
\check{D}_\lambda \big(  e^{-t|z|} D_\lambda f(z)  \big),
\qquad f \in L^2(dw_\lambda),
\]
and for each $\eta \in \{ 0,1 \}^n$ consider an auxiliary Poisson-Bessel-type semigroup
\[
\mathbb{P}_t^{\lambda,\eta,+} f =
\Big( \check{D}_\lambda \big( e^{-t|z|} D_\lambda f^\eta (z) \big) \Big)^+,
\qquad f \in L^2(dw^+_\lambda).
\]
Obviously, by the subordination principle,
\begin{equation}\label{subprince}
\mathbb{P}_t^{\lambda} f (x)=
\int_0^\infty \mathbb{W}_{t^2/(4u)}^{\lambda} f(x) \frac{e^{-u} \, du}{\sqrt{\pi u}},
\qquad
\mathbb{P}_t^{\lambda,\eta,+} f (x)=
\int_0^\infty \mathbb{W}_{t^2/(4u)}^{\lambda,\eta,+} f(x) \frac{e^{-u} \, du}{\sqrt{\pi u}}.
\end{equation}
We focus on the maximal operators, Littlewood-Paley-Stein type mixed $g$-functions and multipliers of Laplace and Laplace-Stieltjes transform type based on these semigroups. In the definitions (1), (2) we exchange $\mathbb{W}_{t}^{\lambda}$ and $\mathbb{W}_{t}^{\lambda,\eta,+}$ with $\mathbb{P}_{t}^{\lambda}$ and $\mathbb{P}_{t}^{\lambda,\eta,+}$, respectively. Further, in (2) we choose $L^2(t^{ 2K + 2|M|  -1}dt)$ instead of
$L^2(t^{ 2K + |M| -1}dt)$ and replace multipliers in (3) and (4) by
\[
|z| \int_0^{\infty} e^{-t|z|} \Phi(t)\, dt
\qquad
\textrm{and}
\qquad
\int_{(0,\infty)} e^{-t|z|} \, d\nu (t),
\]
respectively, leaving the assumptions on $\Phi(t)$ and $\nu$ unchanged.
Then, with the aid of \eqref{subprince}, analogous results to Theorems \ref{thm:main} and \ref{thm:CZ} are valid for these new operators.
The proof is based on a similar procedure to that described in \cite{BCN} or \cite[Section 3]{NoSz}; we leave details to the reader.

The treatment of Lusin area integrals based on the Poisson semigroup is more subtle
and demands more effort and explanation. Consider the following square functions based on the Poisson-Dunkl semigroup and the auxiliary Poisson-Bessel-type counterparts:
\begin{align*}
S^{\lambda}_{P,K,M}(f)(x) =&
\bigg( \int_{\Gamma(x)} t^{2K+2|M|-1} \big|\partial_{t}^K (T^\lambda)^M \mathbb{P}_{t}^{\lambda}f(z)\big|^{2}\frac{dw_{\lambda}(z) \, dt}
{V_{t}^{\lambda}(x)} \bigg)^{1\slash 2},\quad x \in \RR, \\
S^{\lambda,\eta,+}_{P,K,M}(f)(x) =&
\bigg( \int_{\Gamma(x)} t^{2K+2|M|-1} \big|\partial_{t}^K \delta_{\eta,M} \mathbb{P}_{t}^{\lambda,\eta,+}f(z)\big|^{2}
\, \chi_{ \{z \in \R \} }
\frac{dw^+_{\lambda}(z) \, dt}
{V_{t}^{\lambda,+}(x)} \bigg)^{1\slash 2},\quad x \in \R,
\end{align*}
where $M \in \N^n$, $K \in \N$, $|M|+K>0$, and $\Gamma(x)$ is the cone with vertex at $x$,
\[
\Gamma(x) = (x,0) + \Gamma, \qquad
\Gamma = \Big\{(z,t)\in \RR \times(0,\infty) : |z|<t\Big\}.
\]
Our main result concerning these operators reads as follows.
\begin{thm} \label{thm:LAI-Poisson}
Assume that $\lambda \in(-1\slash 2,\infty)^{n}$ and $W$ is a weight on $\RR$ invariant under the reflections $\sigma_1, \ldots ,\sigma_n$.
Then the Lusin area integrals $S^{\lambda}_{P,K,M}$ are bounded on $L^{p}(\RR,Wdw_{\lambda})$, $W^{+}\in A_{p}^{\lambda,+}$, $1<p<\infty$, and from $L^{1}(\RR,Wdw_{\lambda})$ to weak $L^{1}(\RR,Wdw_{\lambda})$, $W^{+}\in A_{1}^{\lambda,+}$.
Moreover, the Poisson-Bessel-type Lusin area integrals $S^{\lambda,\eta,+}_{P,K,M}$, $\eta \in \{ 0,1 \}^n$, are (vector-valued) Calder\'on-Zygmund operators in the sense of the space of homogeneous type $(\R,dw_{\lambda}^+,|\cdot|)$ associated with the Banach spaces
$\mathbb{B} = L^2(\Gamma,t^{2K+2|M|-1}dzdt)$.
\end{thm}
A careful repetition of the arguments justifing Theorems~\ref{thm:main} and \ref{thm:CZ} allows us to reduce proving Theorem~\ref{thm:LAI-Poisson} to showing the standard estimates for the kernels associated with $S^{\lambda,\eta,+}_{P,K,M}$, $\eta \in \{ 0,1 \}^n$,
which are defined as
\[
\mathcal{S}^{\lambda,\eta,+}_{P,K,M}(x,y)
=
\left\{
\partial_{t}^K \delta_{\eta,M,\mathbf{x}} \mathbb{P}_{t}^{\lambda,\eta,+}(\mathbf{x},y)
\Big|_{\mathbf{x} = x+z} \sqrt{\Xi_{\lambda}(x,z,t^2)}
\,\chi_{\{x+z\in\R\}} \right\}_{(z,t) \in \Gamma},
\]
where
\begin{equation}\label{subprinc}
\mathbb{P}_{t}^{\lambda,\eta,+}(x,y)
=
\int_0^\infty \G_{t^2/(4u)}^{\lambda,\eta,+} (x,y) \frac{e^{-u} \, du}{\sqrt{\pi u}}.
\end{equation}
Since it seems not to be straightforward to obtain these estimates from Proposition~\ref{pro:kerest} via \eqref{subprinc}, see the comment in \cite[p.\,1526]{Sz1}, we give the proof of the following result at the end of Section~\ref{sec:kerest}.
\begin{pro}\label{pro:kerestLAIP}
Let $\lambda \in (-1/2, \infty)^n$ and $\eta \in \{ 0 , 1 \}^n$. Then the kernels $\mathcal{S}^{\lambda,\eta,+}_{P,K,M}$ satisfy the standard estimates
with the corresponding Banach spaces
$\mathbb{B} = L^2(\Gamma,t^{2K + 2|M| - 1}dzdt)$. More precisely, they satisfy \eqref{sm1} and \eqref{sm2} with any $\gamma \in (0,1/2]$ such that
$\gamma < \min_{1 \le k \le n} (\lambda_k + 1/2)$.
\end{pro}
We now focus on the relation between Bessel-type objects and the operators associated with the Bessel setting from the papers \cite{BCN,CaSz}.
Note that choosing $\eta_0=(0,\ldots,0)$, in view of symmetry reasons, we have $\mathbb{W}_t^{\lambda,\eta_0,+} = W_t^\lambda$, where $W_t^\lambda$ is the heat-Bessel semigroup considered in the above mentioned papers.
Consequently, many results of \cite{CaSz} can be seen as special cases of Theorem~\ref{thm:CZ} (specified to $\eta = \eta_0$).
However, some definitions in our present situation, in particular when dealing with $\eta = \eta_0$, are a little bit different from the ones used in \cite{BCN,CaSz}.
This concerns especially mixed $g$-functions and higher order Riesz transforms. These new definitions, however, seem to be more natural and appropriate;
see comments in \cite{NoSt5}, where a symmetrization procedure for general orthogonal expansions was proposed and \cite{L}, where this procedure led to the symmetrized Jacobi setting that is connected with the initial context of Jacobi expansions in a similar way as our setting with the Bessel one.

Further, note that Lusin area integrals, which have more complex structure than the mixed $g$-functions, were considered neither in \cite{BCN} nor in \cite{CaSz}.
Consequently, our present results provide some new results in the Bessel setting. This concerns not only
alternatively defined mixed $g$-functions and higher order Riesz transforms, but also mixed Lusin area integrals.
 
Next, we give a comment concerning multipliers of Laplace-Stieltjes transform type.
Using standard arguments, see the comment at the end of \cite[Section 2]{CaSz}, it can be shown that these multipliers and their Bessel-type counterparts are bounded on $L^p(dw_\lambda)$ and $L^p(dw_\lambda^+)$, $1 \le p \le \infty$, respectively.
Here, however, Theorems~\ref{thm:main} and \ref{thm:main+} deliver also weighted $L^p$ mapping properties.

\begin{rem}\label{rem:ext}
The exact aperture of the parabolic cone $A$ is not essential for our developments. Indeed, if we fix $\beta>0$ and write $A_{\beta}=\Big\{(z,t)\in \RR\times(0,\infty) : |z|<\beta\sqrt{t}\Big\}$ instead of $A$ in the definitions of mixed Lusin area integrals, then the results of this paper, and in particular
Theorems \ref{thm:main}-\ref{thm:CZ}, remain valid.
Similar remark concerns the aperture of the cone $\Gamma$ and the results contained in Theorem~\ref{thm:LAI-Poisson}.
\end{rem}

\section{$L^{2}(dw_\lambda^+)$-boundedness} \label{sec:as}

We first collect some auxiliary results which will be needed to establish
Proposition~\ref{pro:L2b}.
In what follows, we will frequently use the fact that the Dunkl transform $D_\lambda$ and its inverse $\check{D}_\lambda$ preserve $\eta$-symmetric functions, namely
\[
D_\lambda \left( L^2(dw_\lambda)_\eta \right) = L^2(dw_\lambda)_\eta, \quad 
\check{D}_\lambda \left( L^2(dw_\lambda)_\eta \right) = L^2(dw_\lambda)_\eta, \qquad 
\eta \in \{0,1\}^n,
\]
where $L^2(dw_\lambda)_\eta = \{ f_\eta : f \in L^2(dw_\lambda)\}$. 

\begin{lem}\label{lem:3.4BCN}
    Let $\lambda \in (-1/2,\infty)^n$, $\eta \in \{0,1\}^n$ and $M \in \N^n$. Then,
    $$\delta_{\eta,M,x}\Big[  (xz)^\eta \varphi_z^{\lambda+\eta}(x) \Big]
        = C_{M,\eta} z^M  (xz)^{\epsilon} \varphi_z^{\lambda+\epsilon}(x), \quad x,z \in \RR,
$$
where
$C_{M,\eta} = (-1)^{\sum_{j=1}^n (1-\eta_j)\overline{M_j} + \lfloor M_j/2\rfloor }$
and $\epsilon = \epsilon(\eta,M) = \eta(\mathbf{1}-\overline{M}) + (\mathbf{1}-\eta)\overline{M}$.
Moreover, $\epsilon \in \mathbb{Z}^n_2$ and
$\epsilon = M \oplus_{\mathbb{Z}_2^n} \eta$.
\end{lem}

\begin{proof}
    By the tensor product structure of $\varphi_z^\lambda$ and $\delta_{\eta,M,x}$
    it is sufficient to analyze the one-dimensional situation $(n=1)$.
Further, since $(xz)^\eta \varphi_z^{\lambda+\eta}(x)$ is
    $\eta$-symmetric with respect to $x$, having in mind \eqref{defpsi} and \eqref{Dunkleigenfv}, it is easy to deduce that
    $$
\delta_{\eta,2,x} \Big[  (xz)^\eta \varphi_z^{\lambda+\eta}(x) \Big]
 =
- \Delta_{\lambda,x} \Big[  (xz)^\eta \varphi_z^{\lambda+\eta}(x) \Big]
        = - z^2 (xz)^\eta \varphi_z^{\lambda+\eta}(x), \qquad x,z \in \mathbb{R}, \quad \eta=0,1.$$
Iterating this identity we see that our task can be reduced to the cases $M=0,1$.
This, however, is a straightforward consequence of the identities
\[
\partial_x \varphi_z^{\lambda}(x) = -z(xz) \varphi_z^{\lambda + 1}(x),
\qquad
\varphi_z^{\lambda}(x) = (2\lambda + 1)\varphi_z^{\lambda + 1}(x)
-(xz)^2 \varphi_z^{\lambda + 2}(x), \quad x,z\in \mathbb{R},
\]
which can easily be verified by means of \cite[(5.3.5)]{Leb} and \cite[(5.3.6)]{Leb}, respectively.
\end{proof}

From the asymptotic behavior of the Bessel function $J_\nu$, $\nu>-1$,
$$J_{\nu}(z) \simeq z^{\nu}, \quad z \to 0^+, \qquad
  J_{\nu}(z) = \mathcal{O}\Big( \frac{1}{\sqrt{z}}\Big), \quad z \to \infty,$$
we can estimate the one-dimensional functions $\varphi_z^{\lambda}$ as follows,
see \cite[Section 2]{CaSz},
\begin{equation}\label{estvarphi}
    |\varphi_z^{\lambda}(x)|
        \lesssim \left\{
            \begin{array}{rl}
                1, & |xz| \leq 1 \\
                |xz|^{-\lambda}, & |xz| > 1
            \end{array} \right., \quad x,z \in \mathbb{R}.
\end{equation}

\begin{lem}\label{lem:3.5BCN}
    Let $\lambda \in (-1/2,\infty)^n$, $\eta \in \{0,1\}^n$ and $f \in L^p(\R,Udw_\lambda^+)$, $U \in A_p^{\lambda,+}$, $1 \le p < \infty$.
    Then, $\mathbb{W}_t^{\lambda,\eta,+} f (x)$
    is a $C^\infty$ function of $(x,t) \in \R \times \mathbb{R}_+$. Moreover, if $K \in \N$ and
    $M \in \N^n$ we have
    $$\partial_t^K \delta_{\eta,M}\mathbb{W}_t^{\lambda,\eta,+} f (x)
        = \int_{\R} \partial_t^K \delta_{\eta,M,x} \G_t^{\lambda,\eta,+} (x,z) f(z) dw_\lambda^+(z),
\qquad x \in \R, \quad t>0.$$
    Further, if $f \in L^2(dw_{\lambda}^+ )$, then for every $x\in \R$ and $t>0$ we have
    $$\partial_t^K \delta_{\eta,M}\mathbb{W}_t^{\lambda,\eta,+} f (x)
        = \int_{\R} \partial_t^K e^{-t|z|^2} \delta_{\eta,M,x}\Big[  (ixz)^\eta \varphi_z^{\lambda+\eta}(x) \Big] D_{\lambda} f^\eta(z) dw_{\lambda}^+(z).$$
    Furthermore, if $f \in C_c^\infty(\R)$, then
$(D_\lambda f^\eta)^+, (\check{D}_\lambda f^\eta)^+ \in C^\infty(\R)$ and we have
\begin{align*}
\delta_{\eta,M} (\check{D}_\lambda f^\eta)^+ (x)
&=
\int_{\R} \delta_{\eta,M,x} \Big[ (ixz)^\eta \varphi_z^{\lambda+\eta}(x) \Big]
f(z) \, dw_{\lambda}^+(z).
\end{align*}
\end{lem}

\begin{proof}
    To prove the first two statements we can proceed as in the proof of
\cite[Lemma 3.5]{BCN}, see also the comments in \cite[Section 2]{CaSz}.
The crucial facts needed in the reasoning are the bounds
$$ \Big|\partial_t^K \delta_{\eta,M,x} \G_t^{\lambda,\eta,+} (x,z)\Big|
        +
     \Big|\partial_t^K e^{-t|z|^2} \delta_{\eta,M,x}\Big[  (xz)^\eta \varphi_z^{\lambda+\eta}(x) \Big]\Big|
        \lesssim e^{-c|z|^2}, \quad z \in \R, \ x \in E, \ t \in F,$$
where $E$ and $F$ are fixed compact subsets of $\R$ and $\mathbb{R}_+$, respectively, and $c>0$ is some constant depending on $E$ and $F$.
The above estimate can be verified by means of Lemma~\ref{lem:EST} below and
Lemma~\ref{lem:3.4BCN} together with the bound \eqref{estvarphi}, respectively.

Finally, the last statement is a consequence of Lemma~\ref{lem:3.4BCN}, the estimate \eqref{estvarphi} and the dominated convergence theorem.
\end{proof}

Now we are ready to give
the proof of $L^2(dw_\lambda^+)$-boundedness of all the investigated Bessel-type operators. This, however, is trivial in the case of multipliers of Laplace and Laplace-Stieltjes transform type, because in both cases
$\mathfrak{m} \in L^\infty(\mathbb{R}^n)$.
Moreover, the $L^2(dw_\lambda^+)$-boundedness of the auxiliary Lusin area integrals
is an immediate consequence of the $L^2(dw_\lambda^+)$-boundedness of the Bessel-type $g$-functions.
Indeed, from the general theory of spaces of homogeneous type we know that
\[
\| S^{\lambda,\eta,+}_{K,M} (f) \|_{L^2(dw_\lambda^+)}
        \simeq \| g^{\lambda,\eta,+}_{K,M} (f) \|_{L^2(dw_\lambda^+)},
\]
which in our context can also be justified with the aid of Lemma \ref{lem:intXi} below.
Thus, it suffices to consider the remaining operators.

\begin{proof}[Proof of Proposition \ref{pro:L2b}; the case of $\mathbb{W}_{*}^{\lambda,\eta,+}$]
    We observe that
    $$\mathbb{W}_{*}^{\lambda,\eta,+}(f)(x)
= x^\eta \mathbb{W}_{*}^{\lambda + \eta, \eta_0,+}\Big(\frac{f}{y^\eta}\Big)(x)
        = x^\eta W_{*}^{\lambda+\eta}\Big(\frac{f}{y^\eta}\Big)(x), \quad x \in \R,$$
where $\eta_0 =(0,\ldots,0)$ and $W_{*}^{\lambda}$ is the heat-Bessel semigroup maximal operator, see the comments following the statement of Proposition~\ref{pro:kerestLAIP}.
    Since $W_{*}^{\lambda}$ is bounded on $L^2(dw^+_{\lambda})$, see \cite[Section 4.1]{BCN} and \cite[Section 2]{CaSz},
    we easily deduce that
    \begin{align*}
        \left\| \mathbb{W}_{*}^{\lambda,\eta,+}(f) \right\|_{L^2(dw^+_{\lambda})}
            = \Big\| W_{*}^{\lambda+\eta}\Big(\frac{f}{y^\eta}\Big) \Big\|_{L^2(dw^+_{\lambda+\eta})}
            \lesssim \Big\| \frac{f}{y^\eta} \Big\|_{L^2(dw^+_{\lambda+\eta})}
            = \| f \|_{L^2(dw^+_{\lambda})}.
    \end{align*}
\end{proof}

\begin{proof}[Proof of Proposition \ref{pro:L2b}; the case of $g^{\lambda,\eta,+}_{K,M}$]
    Let $f \in L^2(dw_\lambda^+)$. By  Lemmas~\ref{lem:3.5BCN} and \ref{lem:3.4BCN} we can write, for every $x \in \R$ and $t>0$,
    \begin{align*}
        \Big| \partial_t^K \delta_{\eta,M,x} \mathbb{W}_t^{\lambda,\eta,+}f(x) \Big|
=& \Big|
\int_{\R} |z|^{2K} e^{-t|z|^2} z^M (xz)^\epsilon \varphi_z^{\lambda +\epsilon}(x) D_\lambda f^\eta (z) \, dw_{\lambda}^+(z)  \Big|
\\
=& \,
2^{-n} \Big| \int_{\RR} |z|^{2K} e^{-t|z|^2} z^M \psi_z^{\lambda}(x) D_\lambda f^\eta (z) \, dw_{\lambda}(z) \Big| \\
=&
\Big|\check{D}_\lambda
                \Big[ |z|^{2K} e^{-t|z|^2} z^{M}
                    D_\lambda f^\eta (z)\Big](x)\Big|,
    \end{align*}
where the second equality holds by symmetry reasons.
Since the function in square brackets above is $\epsilon$-symmetric and $D_{\lambda}$ is an isometry in $L^2(dw_\lambda)$, we obtain
    \begin{align*}
        \left\| g^{\lambda,\eta,+}_{K,M}(f) \right\|_{L^2(dw^+_{\lambda})}^2
         = &
\, 2^{-n} \int_{0}^\infty \int_{\RR}
\Big|\check{D}_\lambda
                \Big[ |z|^{2K} e^{-t|z|^2} z^{M}
                    D_\lambda f^\eta (z)\Big](x)\Big|^2
t^{2K+|M|-1}  dt \, dw_{\lambda}(x) \\
\lesssim &
\int_{\RR} |z|^{4K+2|M|} \big| D_{\lambda} f^\eta(z) \big|^2
\int_0^\infty e^{-2t|z|^2} t^{2K+|M|-1} dt \, dw_{\lambda}(z) \\
        \simeq & \, \| D_{\lambda} f^\eta \|_{L^2(dw_{\lambda})}^2
        = \|f^\eta\|_{L^2(dw_\lambda)}^2
        \simeq \|f\|_{L^2(dw_\lambda^+)}^2,
    \end{align*}
which finishes the reasoning.
\end{proof}

\begin{proof}[Proof of Proposition \ref{pro:L2b}; the case of $R_M^{\lambda,\eta,+}$]
     Take $f \in  C^{\lambda,\eta,+}$.
Observe that $( |z|^{-|M|} D_\lambda f^\eta (z))^+ \in C_c^\infty(\R)$ and thus an application of
Lemma~\ref{lem:3.5BCN} together with Lemma~\ref{lem:3.4BCN} gives
\begin{align*}
\big| R_M^{\lambda,\eta,+} f (x) \big|
        &= \Big| \int_{\R} |z|^{-|M|} z^M (xz)^{\epsilon} \varphi_z^{\lambda+\epsilon}(x)  D_{\lambda} f^\eta(z) dw_{\lambda}^+(z) \Big| \\
&=
\, 2^{-n} \Big| \int_{\RR} |z|^{-|M|} z^M \psi_z^{\lambda}(x)  D_{\lambda} f^\eta(z) dw_{\lambda}(z) \Big| \\
&=
\Big|\check{D}_\lambda
                \Big[ |z|^{-|M|} z^{M}
                    D_\lambda f^\eta (z)\Big](x)\Big|.
\end{align*}
Since the last expression, as a function on $\RR$, is symmetric, using the $L^2(dw_\lambda)$-isometry of $D_\lambda$ we get
     \begin{align*}
        \big\| R_M^{\lambda,\eta,+}(f) \big\|_{L^2(dw^+_{\lambda})}
          \simeq \left\|  |z|^{-|M|} z^{M}
         D_{\lambda} f^\eta(z) \right\|_{L^2(dw_{\lambda})}
         \leq \| D_{\lambda} f^\eta \|_{L^2(dw_{\lambda})}
        \simeq \|f\|_{L^2(dw_\lambda^+)}.
     \end{align*}
     Thus $R_M^{\lambda,\eta,+}$ is bounded from $C^{\lambda,\eta,+}$ into $L^2(dw_\lambda^+)$.
Next, we ensure that $C^{\lambda,\eta,+}$ is a dense subspace of $L^2(dw_\lambda^+)$.
It is not hard to verify that
\[
C^{\lambda,\eta,+} =
\check{D}_{\lambda}\big( \big( C_c^\infty(\R) \big)^\eta \big)^+,
\]
where $\big(C_c^\infty(\R) \big)^\eta = \{ f^\eta : f \in C_c^\infty(\R) \}$.
Since $C_c^\infty(\R)$ is dense in $L^2(dw_{\lambda}^+)$, it follows that $\big( C_c^\infty(\R) \big)^\eta$ is dense in $L^2(dw_\lambda)_\eta$ 
and the conclusion follows.
\end{proof}

\section{Kernel estimates}\label{sec:kerest}
In this section we gather various technical facts, which finally allow us to obtain standard estimates for all the kernels under consideration.
Our method is based on the technique used in the Bessel context in the paper \cite{BCN} for the restricted range of
$\lambda \in [0,\infty)^n$ and in \cite{CaSz} for the full range of $\lambda \in (-1\slash 2,\infty)^n$.

Given $\eta \in \{0,1 \}^n$, the Bessel-type heat kernel can be expressed as, see \cite[(3.3)]{CaSz},
\begin{equation}\label{Bhk}
\G_t^{\lambda,\eta,+}(x,y)= \sum_{\eps \in \{0,1\}^n} C_{\lambda + \eta,\eps} \,
    t^{-n\slash 2-|\lambda| - |\eta| -2|\eps|} (xy)^{2\eps + \eta} \int\exp\left(-\frac{q(x,y,s)}{4t}\right) \,
d\Omega_{\lambda+\eta+ \mathbf{1} + \eps }(s),
\end{equation}
where $C_{\lambda,\eps}=(2\lambda+\mathbf{1})^{\mathbf{1}-\eps} \,2^{-n\slash 2 - |\lambda| - 2|\eps|}$, and the function $q(x,y,s)$ is defined as
$$
q(x,y,s) = |x|^2 + |y|^2 + 2 \sum_{j=1}^n x_j y_j s_j, \qquad x,y \in \mathbb{R}_+^n, \quad s \in [-1,1]^n.
$$
The measures $\Omega_\nu$, $\nu \in (0,\infty)^n$, appearing in \eqref{Bhk} are products of one-dimensional
measures, $\Omega_\nu = \bigotimes_{j=1}^{n} \Omega_{\nu_j}$, where $\Omega_{\nu_j}$ is defined on the interval $[-1,1]$ by the density
$$
 d \Omega_{\nu_j}(s_j)
    = \frac{ (1-s_j^2)^{\nu_j - 1} ds_j } {\sqrt{\pi} 2^{\nu_j - 1\slash 2}
\Gamma{ (\nu_{j}) } }.
$$
To shorten notation, we will write $\q$ instead of $q(x,y,s)$ when no confusion can arise.

To estimate kernels defined via $\G_t^{\lambda,\eta,+}(x,y)$ we will frequently use the following.

\begin{lem}\label{lem:EST}
    Let $\lambda \in (-1/2,\infty)^n$, $\eta \in \{ 0, 1 \}^n$, $\ell, r, M \in \N^n$ and $K \in \N$. Then
    \begin{align*}
        & \big| \partial^K_t \partial_x^\ell \partial_y^r \delta_{\eta,M,x} \G_t^{\lambda,\eta,+}(x,y)\big| \\
        & \quad \lesssim \sum_{\substack{ \eps, \zeta, \rho \in \{0,1\}^n \\
\alpha, \beta  \in \{0,1,2\}^n  }} x^{2\eps - \alpha\eps + \eta - \zeta\eta} y^{2\eps-\beta \eps + \eta - \rho \eta}
    t^{-(n/2+|\lambda|+ |\eta| + 2|\eps|)-K-(|M| + |\ell| +|r|)/2 + (|\alpha \eps| + |\zeta \eta| + |\beta \eps| + |\rho \eta|)/2} \\
& \qquad \times
                \int \exp\left(-\frac{\q}{8t}\right)
d\Omega_{\lambda+\eta+ \mathbf{1} + \eps }(s),
    \end{align*}
    uniformly in $x,y \in \R$ and $t>0$.
\end{lem}

\begin{proof}
We first deal with the one-dimensional situation ($n=1$). Since the heat-Dunkl semigroup $\{\mathbb{W}_t^\lambda\}_{ t>0}$ satisfies the heat equation
$\partial_t \mathbb{W}_t^\lambda = - \Delta_{\lambda} \mathbb{W}_t^\lambda$, it can easily be verified that its kernel $\mathbb{G}_t^\lambda (x,y)$ also satisfies this equation with respect to $x$ variable, i.e.
\[
\partial_t \mathbb{G}_t^\lambda (x,y)
= - \Delta_{\lambda,x} \mathbb{G}_t^\lambda (x,y)
= (T_{x}^\lambda)^2 \mathbb{G}_t^\lambda (x,y).
\]
Observe that for each $\eta =0,1$ the functions
$\partial_t \mathbb{G}_t^{\lambda,\eta,+} (x,y)$ and
$\delta_{\eta,2,x} \mathbb{G}_t^{\lambda,\eta,+} (x,y)$ are $\eta$-symmetric with respect to $x$. Thus, in view of the above identity, they are both $\eta$-symmetric component of $2^n \partial_t \mathbb{G}_t^\lambda (x,y)$ and hence they are equal.
Iterating this identity we obtain
\begin{equation*}
\delta_{\eta,M,x} \mathbb{G}_t^{\lambda,\eta,+} (x,y)
    =
    \partial_t^{\lfloor M/2 \rfloor}
    \left( \partial_x^{\overline{M}} + \frac{2\lambda \eta \overline{M}}{x} \right)
    \mathbb{G}_t^{\lambda,\eta,+} (x,y), \qquad x,y,t>0, \quad M \in \N, \quad \eta=0,1.
\end{equation*}
Combining this with the representation formula \eqref{Bhk} gives
\begin{align*}
        & \left|\partial_t^{K}\partial_{x}^{\ell} \partial_{y}^{r} \delta_{\eta,M,x} \G_t^{\lambda,\eta,+}(x,y) \right|
               = \left|\partial_t^{K+\lfloor M/2 \rfloor}\partial_{x}^{\ell} \partial_{y}^{r} \left( \partial_{x}^{\overline{M}} + \frac{2\lambda \eta \overline{M}}{x} \right)  \G_t^{\lambda,\eta,+}(x,y) \right| \\
            & \qquad \lesssim \sum_{\eps = 0,1} \int \Big\{
                    \left| \partial_t^{K+\lfloor M/2 \rfloor}
\partial_{x}^{\ell+ \overline{M}} \partial_{y}^{r} \left[ t^{-(1/2+\lambda+\eta+2\eps)} (xy)^{2\eps+\eta} \exp\left(-\frac{\q}{4t}\right)\right] \right|   \\
        & \qquad \qquad + \chi_{ \{ \eta = \overline{M}=1 \}}
\left| \partial_t^{K+\lfloor M/2 \rfloor}\partial_{x}^{\ell} \partial_{y}^{r}
\left[ t^{-(1/2+\lambda + 1 +2\eps)} x^{2\eps}y^{2\eps+1} \exp\left(-\frac{\q}{4t}\right)\right] \right|
                \Big\} d \Omega_{\lambda+\eta+1+\eps}(s).
\end{align*}
Proceeding in a similar way as in the proof of \cite[Lemma 3.2]{CaSz}, we obtain
\begin{align*}
& \left| \partial_t^W \partial_x^L \partial_y^R
\left[ t^S x^{2\eps +\eta_1} y^{2\eps +\eta_2} \exp\left(-\frac{\q}{4t}\right) \right]\right| \\
& \qquad \lesssim
\sum_{\substack{ \zeta, \rho = 0,1 \\
\alpha, \beta  = 0,1,2  }} x^{2\eps - \alpha\eps + \eta_1 - \zeta\eta_1} y^{2\eps-\beta \eps + \eta_2 - \rho \eta_2}
    t^{S-W-(L+R)/2 + (\alpha \eps + \zeta \eta_1 + \beta \eps + \rho \eta_2)/2}
                \exp\left(-\frac{\q}{8t}\right),
\end{align*}
uniformly in $x,y,t>0$; here $W,L,R \in \N$, $S\in \mathbb{R}$ and
$\eps,\eta_1,\eta_2 = 0,1$ are fixed.
Using this estimate we see that
\begin{align*}
        &  \left| \partial_t^{K+\lfloor M/2 \rfloor}
\partial_{x}^{\ell+ \overline{M}} \partial_{y}^{r} \left[ t^{-(1/2+\lambda+\eta+2\eps)} (xy)^{2\eps+\eta} \exp\left(-\frac{\q}{4t}\right)\right] \right| \\
        & \qquad \lesssim \sum_{\substack{ \zeta, \rho = 0,1 \\
\alpha, \beta  = 0,1,2  }} x^{2\eps - \alpha\eps + \eta - \zeta\eta} y^{2\eps-\beta \eps + \eta - \rho \eta}
t^{-(1/2+\lambda+\eta+2\eps)-K -
(M+ \ell + r)/2 +(\alpha \eps + \zeta\eta + \beta \eps + \rho \eta)/2} 
\exp\left(-\frac{\q}{8t}\right),
    \end{align*}
    and
    \begin{align*}
        & \chi_{ \{ \eta = \overline{M}=1 \}}
\left| \partial_t^{K+\lfloor M/2 \rfloor}\partial_{x}^{\ell} \partial_{y}^{r}
\left[ t^{-(1/2+\lambda + 1 +2\eps)} x^{2\eps}y^{2\eps+1} \exp\left(-\frac{\q}{4t}\right)\right] \right| \\
        & \qquad  \lesssim \chi_{ \{ \eta =1 \}}
\sum_{\substack{  \rho = 0,1 \\
\alpha, \beta  = 0,1,2  }} x^{2\eps - \alpha\eps}
y^{2\eps-\beta \eps + 1 - \rho }
t^{-(1/2+\lambda + 1 + 2\eps)-K - (M + \ell + r)/2 + (\alpha \eps + 1 +
\beta \eps + \rho)/2} \exp\left(-\frac{\q}{8t}\right).
\end{align*}
Since the bound of the first expression is dominating, we arrive at the desired conclusion.

Now suppose that $n$ is arbitrary. In view of the product structure of $\G_t^{\lambda,\eta,+}(x,y)$, an application of Leibniz' rule gives
    \begin{align*}
         \partial^K_t \partial_x^\ell \partial_y^r \delta_{\eta,M,x} \G_t^{\lambda,\eta,+}(x,y)
            = & \partial^K_t \left[ \prod_{j=1}^n  \partial_{x_j}^{\ell_j} \partial_{y_j}^{r_j} \delta_{j,\eta_j,M_j,x_j} \G_t^{\lambda_j,\eta_j,+}(x_j,y_j)\right]  \\
            = & \sum_{\substack{ k \in \{0, \dots, K\}^n \\  |k| = K }}
                               c_k \prod_{j=1}^n  \left[\partial_t^{k_j}\partial_{x_j}^{\ell_j} \partial_{y_j}^{r_j} \delta_{j,\eta_j,M_j,x_j} \G_t^{\lambda_j,\eta_j,+}(x_j,y_j) \right],
\end{align*}
which leads directly to the asserted bound.
\end{proof}

The next result constitutes an essence of our method. It provides
a link from the estimates emerging from the integral
representation of $\G_t^{\lambda,\eta,+}(x,y)$, see \eqref{Bhk}
and Lemma~\ref{lem:EST}, to the standard estimates related to the
space of homogeneous type $(\R,dw_\lambda^+, |\cdot|)$. It may be
proved in much the same way as \cite[Lemma 3.3]{CaSz}. The crucial
role in the proof plays \cite[Lemma 2.1]{NoSz}, see also
\cite[Lemma 3.1]{CaSz}.

\begin{lem}\label{lem:EST2}
    Assume that $\lambda \in (-1\slash 2 , \infty)^{n}$, $1 \le p \le \infty$, $W \in \mathbb{R}$, $C>0$,
    $\eps, \eta, \zeta, \rho \in \{ 0,1 \}^{n}$ and $\alpha,\beta \in \{ 0,1,2 \}^{n}$.
    Further, let $\tau \in \mathbb{N}^n$ be such that $\tau \le 2\eps - \alpha \eps + \eta - \zeta \eta$.
    Given $u \ge 0$, we consider the function $\Upsilon_u \colon \R \times \R \times \mathbb{R}_+ \to \mathbb{R}$ defined by
    \begin{align*}
        \Upsilon_{u}(x,y,t)
            = & x^{2\eps-\alpha \eps+\eta-\zeta \eta - \tau}
                \, y^{2\eps-\beta \eps+\eta-\rho \eta}
                \, t^{ -(n \slash 2 + |\lambda| +|\eta| + 2|\eps|) + (|\alpha \eps|+|\zeta \eta| + |\beta \eps|+ |\rho \eta| + |\tau|)\slash 2 - W\slash p-u\slash 2} \\
              & \times \int  \exp\Big(-\frac{C\q}{t} \Big) \, d\Omega_{\lambda+\eta+\mathbf{1}+\eps}(s),
    \end{align*}
    where $W \slash p = 0$ for $p=\infty$. Then $\Upsilon_u$ satisfies the integral estimate
    \begin{align*}
        \big\|\Upsilon_{u}\big(x,y,t\big)\big\|_{L^{p}(t^{W-1}dt)}
        \lesssim \frac{1}{|x-y|^{u}} \;
        \frac{1}{w^+_{\lambda}(B(x,|y-x|))},
    \end{align*}
    uniformly in $x,y \in \R$, $x\neq y$.
\end{lem}

It should be noted that for every $\lambda \in (-1\slash 2,\infty)^n$
the $w^+_{\lambda}$ measure of the ball $B(x,R)$ or the cube centered at $x$ and of side lengths $2R$ can be described as
\begin{equation}\label{ball}
w^+_{\lambda}(B(x,R)) \simeq V_R^{\lambda,+}(x) \simeq R^n \prod_{j=1}^n (x_j + R)^{2\lambda_j},
 \qquad x \in \R, \quad R>0.
\end{equation}

The two lemmas below will be useful in justifying the smoothness estimates \eqref{sm1} and \eqref{sm2} when the corresponding kernel is not scalar valued
($\B \ne \mathbb{C}$).

\begin{lem}[{\cite[Lemma 4.3]{Sz1}}] \label{lem:theta}
Let $x,y,z\in\R$ and $s \in [-1,1]^n$. Then
$$
\frac{1}{4} q(x,y,s) \le q(z,y,s) \le 4 q(x,y,s),
$$
provided that $|x-y|>2|x-z|$. Similarly, if $|x-y|>2|y-z|$ then
$$
\frac{1}{4} q(x,y,s) \le q(x,z,s) \le 4 q(x,y,s).
$$
\end{lem}

\begin{lem}[{\cite[Lemma 4.5]{Sz1}}] \label{lem:double}
Let $\lambda \in (-1 \slash 2,\infty)^n$ and $\gamma \in \mathbb{R}$ be fixed. We have
\begin{align*}
\bigg( \frac{1}{|z-y|} \bigg)^{\gamma}
\frac{1}{w^+_{\lambda}(B(z,|z-y|))}
\simeq
\bigg( \frac{1}{|x-y|} \bigg)^{\gamma}
\frac{1}{w^+_{\lambda}(B(x,|x-y|))}
\end{align*}
on the set $\{(x,y,z) \in \R \times \R \times \R : |x-y|>2|x-z|\}$.
\end{lem}

To be precise, in \cite[Lemma 4.5]{Sz1} only the restricted range of $\lambda \in [0,\infty)^n$ was allowed.
However, the same arguments as those in \cite{Sz1} show the result in the general case.

The next lemmas will be useful when proving kernel estimates for the kernels associated with Lusin area integrals.

\begin{lem}[{\cite[Lemma 4.7]{Sz1}}] \label{lem:qz}
Let $x,y\in\R$, $z\in\RR$, $s\in[-1,1]^{n}$. Then
\begin{align*}
q(x+z,y,s) \geq \frac{1}{2}q(x,y,s)-|z|^{2}.
\end{align*}
\end{lem}

\begin{lem}\label{lem:intXi}
Let $\lambda \in (-1/2,\infty)^n$. Then
\[
\int_{|z|<\sqrt{t}} \Xi_{\lambda}(x,z,t) \chi_{ \{ x+z \in \R \} } \, dz
\simeq 1,
\]
uniformly in $x \in \R$ and $t>0$.
\end{lem}

\begin{proof}
Since
\[
\int_{|z|<\sqrt{t}} \Xi_{\lambda}(x,z,t) \chi_{ \{ x+z \in \R \} } \, dz
=
\frac{1}{V_{\sqrt{t}}^{\lambda,+}(x)} \int_{|z|<\sqrt{t}} (x+z)^{2\lambda}
\chi_{ \{ x+z \in \R \} } \, dz
=
\frac{w_{\lambda}^+(B(x,\sqrt{t}))}{V_{\sqrt{t}}^{\lambda,+}(x)},
\]
the conclusion is a straightforward consequence of
\eqref{ball}.
\end{proof}

The result below can be seen as an extension of \cite[Lemma 4.8]{Sz1}, which in our notation is valid only for $\lambda \in [0,\infty)^n$.
The crucial role in the proof plays the estimate
\begin{equation}\label{xiineq}
|x - y|^\xi \lesssim
|x^\xi - y^\xi|, \qquad x,y \ge 0,
\end{equation}
where $\xi \ge 1$ is fixed. This was also used in \cite[p.\,1540]{Sz1} but only with $\xi =2$.
Here the technical side demands more effort and seems to be unavoidable in the general case
$\lambda \in (-1/2,\infty)^n$.
\begin{lem}\label{lem:intdifXi}
Given $\lambda \in (-1/2,\infty)^n$, there exists $\gamma=\gamma(\lambda) \in (0,1/2]$ such that
\[
\int_{|z|<\sqrt{t}}
\chi_{ \{ x+z, \, x'+z \in \R  \} }
\big| \sqrt{ \Xi_\lambda(x,z,t) } - \sqrt{ \Xi_\lambda(x',z,t) }
\big|^2  \, dz
\lesssim
\bigg( \frac{|x-x'|^2}{t} \bigg)^{\gamma},
\]
uniformly in $x,x' \in \R$ and $t>0$.
Moreover, one can take any $\gamma \in (0,1/2]$ satisfying
$\gamma < \min_{1\le k \le n}(\lambda_k +1/2)$.
\end{lem}

\begin{proof}
Let $\gamma \in (0,1/2]$ satisfying
$\gamma < \min_{1\le k \le n}(\lambda_k +1/2)$ be fixed.
Using \eqref{xiineq} with $\xi = 2$, we see that in order to finish the proof, it suffices to verify that
\begin{align}\label{red11}
\int_{|z|<\sqrt{t}}
\chi_{ \{ x+z, \,  x'+z \in \R  \} }
\big| \Xi_\lambda(x,z,t)  - \Xi_\lambda(x',z,t)  \big| \, dz
\lesssim
\bigg( \frac{|x-x'|^2}{t} \bigg)^{\gamma}.
\end{align}
Further, we may  reduce our task to showing the one-dimensional version of \eqref{red11}.
Indeed, taking into account the product structure of $\Xi_\lambda(x,z,t)$ and the identity
\[
\prod_{j=1}^n a_j - \prod_{j=1}^n b_j
=
\sum_{k=1}^n \left( \prod_{j=1}^{k-1} b_j \right) (a_k - b_k)
\left( \prod_{j=k+1}^{n} a_j \right),
\qquad a_j,b_j \in \mathbb{R}, j=1,\ldots,n
\]
(here we use the standard notation concerning empty products) we get
\begin{align*}
& \big| \Xi_\lambda(x,z,t)  - \Xi_\lambda(x',z,t)  \big| \\
& \quad \le
\sum_{k=1}^n \left( \prod_{j=1}^{k-1} \Xi_{\lambda_j}(x'_j,z_j,t) \right)
\left( \prod_{j=k+1}^{n} \Xi_{\lambda_j}(x_j,z_j,t) \right)
\big| \Xi_{\lambda_k}(x_k,z_k,t)  - \Xi_{\lambda_k}(x'_k,z_k,t)  \big|.
\end{align*}
This together with the one-dimensional versions of Lemma~\ref{lem:intXi} and \eqref{red11} gives
\begin{align*}
& \int_{|z|<\sqrt{t}}
\chi_{ \{ x+z, \, x'+z \in \R  \} }
\big| \Xi_\lambda(x,z,t)  - \Xi_\lambda(x',z,t)  \big| \, dz \\
& \quad \le
\sum_{k=1}^n \left( \prod_{j=1}^{k-1} \int_{|z_j| < \sqrt{t}}
\chi_{ \{ x'_j+z_j >0  \} } \Xi_{\lambda_j}(x'_j,z_j,t) \, dz_j \right)
\left( \prod_{j=k+1}^{n} \int_{|z_j| < \sqrt{t}}  \chi_{ \{ x_j+z_j >0  \}} \Xi_{\lambda_j}(x_j,z_j,t) \, dz_j \right)\\
& \qquad \qquad \qquad \times
\left( \int_{|z_k| < \sqrt{t}} \chi_{ \{ x_k+z_k, \, x'_k+z_k >0  \} }
\big| \Xi_{\lambda_k}(x_k,z_k,t)  - \Xi_{\lambda_k}(x'_k,z_k,t)  \big| \, dz_k \right)\\
& \quad \lesssim
\sum_{k=1}^n \bigg( \frac{|x_k-x'_k|^2}{t} \bigg)^{\gamma}
\lesssim
\bigg( \frac{|x-x'|^2}{t} \bigg)^{\gamma}.
\end{align*}

Next, we show \eqref{red11} for $n=1$.
We can assume that $|x-x'| \le \sqrt{t}$, since otherwise
\eqref{red11} is a straightforward consequence of Lemma \ref{lem:intXi} and the inequality $1 \le \Big( \frac{|x-x'|^2}{t} \Big)^{\gamma}$.
Further, observe that
\begin{align*}
& \chi_{ \{ x+z, \, x'+z > 0 \} }
\big| \Xi_{\lambda}(x,z,t)  - \Xi_{\lambda}(x',z,t)  \big| \\
& \quad \le
\chi_{ \{ x+z, \, x'+z > 0 \} }
\frac{\left| (x+z)^{2\lambda} - (x'+z)^{2\lambda} \right| }
{V_{\sqrt{t}}^{\lambda,+}(x')}
+
\chi_{ \{ x+z, \, x'+z > 0 \} } \Xi_{\lambda}(x,z,t)
\frac{\left| V_{\sqrt{t}}^{\lambda,+}(x) - V_{\sqrt{t}}^{\lambda,+}(x') \right|}
{V_{\sqrt{t}}^{\lambda,+}(x')} \\
& \quad \equiv
I_1(x,x',z,t) + I_2(x,x',z,t).
\end{align*}
We will treat $I_1$ and $I_2$ separately. We first deal with $I_1$. Using \eqref{xiineq} specified to $\xi=1/(2\gamma)$ and then the Mean Value Theorem we arrive at the estimates
\begin{align*}
V_{\sqrt{t}}^{\lambda,+}(x') I_1(x,x',z,t)
& \lesssim
\chi_{ \{ x+z, \, x'+z > 0 \} }
\left| (x+z)^{\lambda/\gamma} - (x'+z)^{\lambda/\gamma} \right|^{2\gamma} \\
& \lesssim
\chi_{ \{ x+z, \, x'+z > 0 \} }
|x-x'|^{2\gamma} (\theta + z)^{2(\lambda - \gamma)},
\end{align*}
where $\theta$ is a convex combination of $x$ and $x'$, which may depend on $z$ and $t$. We denote
\[
y=\left\{ \begin{array}{ll}
x \vee x', & \lambda \ge \gamma \\
x \wedge x', & \lambda < \gamma
\end{array} \right..
\]
Then obviously
\begin{align*}
\int_{|z|<\sqrt{t}}  I_1(x,x',z,t) \, dz
\lesssim
\frac{|x-x'|^{2\gamma}}
{V_{\sqrt{t}}^{\lambda,+}(x')}
\int_{|z|<\sqrt{t}} \chi_{ \{ y+z > 0 \} }
(y + z)^{2(\lambda - \gamma)} \, dz
=
|x-x'|^{2\gamma} \frac{w^+_{\lambda - \gamma}(B(y,\sqrt{t}))}
{V_{\sqrt{t}}^{\lambda,+}(x')}.
\end{align*}
Combining this with \eqref{ball} and the relations
\begin{equation}\label{comp111}
(x + \sqrt{t}) \simeq (x' + \sqrt{t}) \simeq (x \wedge x' + \sqrt{t}) \simeq (x \vee x' + \sqrt{t}),
\end{equation}
valid when $|x-x'| \le \sqrt{t}$,
we get
\[
\int_{|z|<\sqrt{t}}  I_1(x,x',z,t) \, dz
\lesssim
|x-x'|^{2\gamma} \frac{(y+\sqrt{t})^{2(\lambda - \gamma)}}{(x'+\sqrt{t})^{2\lambda}}
\lesssim
\bigg( \frac{|x-x'|^2}{t} \bigg)^{\gamma}.
\]

We now focus on $I_2$.
By Lemma~\ref{lem:intXi}, an explicit expression of $V_{\sqrt{t}}^{\lambda,+}(x')$ (see \eqref{Vexplicit}) and \eqref{ball} we have
\begin{align*}
& \int_{|z|<\sqrt{t}}  I_2(x,x',z,t) \, dz
\lesssim
\frac{\left| V_{\sqrt{t}}^{\lambda,+}(x) - V_{\sqrt{t}}^{\lambda,+}(x') \right|}
{V_{\sqrt{t}}^{\lambda,+}(x')} \\
    & \quad \lesssim
\frac{\left| (x + \sqrt{t})^{2\lambda + 1} - (x' + \sqrt{t})^{2\lambda + 1} \right|}
{ \sqrt{t} (x' + \sqrt{t})^{2\lambda} }
+
\frac{\left| \chi_{ \{ x > \sqrt{t} \} }(x - \sqrt{t})^{2\lambda + 1} -
\chi_{ \{ x' > \sqrt{t} \} } (x' - \sqrt{t})^{2\lambda + 1} \right|}
{ \sqrt{t} (x' + \sqrt{t})^{2\lambda} } \\
    & \quad \equiv J_1(x,x',t) + J_2(x,x',t).
\end{align*}
Thus to complete the proof it is enough to check that
\[
J_j(x,x',t) \lesssim
\frac{|x-x'|}{\sqrt{t}}
+
\bigg( \frac{|x-x'|^2}{t} \bigg)^{\lambda + 1/2}
+
\bigg( \frac{|x-x'|^2}{t} \bigg)^{\gamma},
\qquad |x-x'| \le \sqrt{t}, \quad j=1,2;
\]
notice that in this sum the last term is dominating.
We first focus on $J_1$. By means of the Mean Value Theorem and \eqref{comp111}
we obtain
\[
J_1(x,x',t) \lesssim \frac{|x - x'| (\theta + \sqrt{t})^{2\lambda}}
{\sqrt{t} (x' + \sqrt{t})^{2\lambda}}
\simeq \frac{|x - x'|}{\sqrt{t}},
\]
where $\theta$ is a convex combination of $x$ and $x'$. To treat $J_2$ we observe that
\begin{align*}
J_2(x,x',t) &=
\chi_{ \{ x \wedge x' \le \sqrt{t} < x \vee x' \} }
\frac{(x \vee x' - \sqrt{t})^{2\lambda + 1}}
{\sqrt{t} (x' + \sqrt{t})^{2\lambda}} \\
& \quad +
\left( \chi_{ \{ x \wedge x' \ge 2\sqrt{t} \} }
+ \chi_{ \{ 2\sqrt{t} > x \wedge x' > \sqrt{t} \} } \right)
\frac{\left| (x-\sqrt{t})^{2\lambda + 1} - (x'-\sqrt{t})^{2\lambda + 1} \right| }
{\sqrt{t} (x' + \sqrt{t})^{2\lambda}}.
\end{align*}
The relevant estimate for the first term is straightforward because, in view of
\eqref{comp111}, we have
\[
\chi_{ \{ x \wedge x' \le \sqrt{t} < x \vee x' \} }
\frac{(x \vee x' - \sqrt{t})^{2\lambda + 1}}
{\sqrt{t} (x' + \sqrt{t})^{2\lambda}}
\lesssim
\frac{(x \vee x' - x \wedge x')^{2\lambda + 1}}{(\sqrt{t})^{2\lambda + 1}}
=
\bigg( \frac{|x-x'|^2}{t} \bigg)^{\lambda + 1/2}.
\]
To bound the second one we apply the Mean Value Theorem to obtain
\begin{align*}
\chi_{ \{ x \wedge x' \ge 2\sqrt{t} \} }
\left| (x-\sqrt{t})^{2\lambda + 1} - (x'-\sqrt{t})^{2\lambda + 1} \right|
    & \simeq
\chi_{ \{ x \wedge x' \ge 2\sqrt{t} \} }
| x-x' | (\theta-\sqrt{t})^{2\lambda},
\end{align*}
where $\theta$ is a convex combination of $x$ and $x'$.
Since for $x\wedge x' \ge 2\sqrt{t}$ the expression $\theta - \sqrt{t}$ is comparable to quantities appearing in \eqref{comp111}, we get the desired estimate.
Treating the last term in a similar way as $I_1$ at the beginning of the proof, we get
\begin{align*}
\chi_{ \{ 2\sqrt{t} > x \wedge x' > \sqrt{t} \} }
\left| (x-\sqrt{t})^{2\lambda + 1} - (x'-\sqrt{t})^{2\lambda + 1} \right|
\lesssim
\chi_{ \{ 2\sqrt{t} > x \wedge x' > \sqrt{t} \} }
|x - x'|^{2\gamma}
 (\theta-\sqrt{t})^{2\lambda + 1 - 2\gamma},
\end{align*}
where $\theta$ is a convex combination of $x$ and $x'$.
Since $2\lambda +1 -2\gamma >0$,
in view of \eqref{comp111} we have
\begin{align*}
\chi_{ \{ 2\sqrt{t} > x \wedge x' > \sqrt{t} \} }
\frac{\left| (x-\sqrt{t})^{2\lambda + 1} - (x'-\sqrt{t})^{2\lambda + 1} \right| }
{\sqrt{t} (x' + \sqrt{t})^{2\lambda}}
& \lesssim
\chi_{ \{ 2\sqrt{t} > x \wedge x' > \sqrt{t} \} }
|x - x'|^{2\gamma}
\frac{ (x \vee x' + \sqrt{t})^{2\lambda + 1 - 2\gamma} }
{\sqrt{t} (x' + \sqrt{t})^{2\lambda}} \\
& \simeq
\chi_{ \{ 2\sqrt{t} > x \wedge x' > \sqrt{t} \} }
\bigg( \frac{|x-x'|^2}{t} \bigg)^{\gamma}.
\end{align*}
This finishes the proof of Lemma~\ref{lem:intdifXi}.
\end{proof}

\begin{lem}\label{lem:intXi2}
Let $\lambda \in (-1/2,\infty)^n$ be fixed. Then there exists $\gamma=\gamma(\lambda) \in (0,1/2]$ such that
\[
\int_{|z|<\sqrt{t}} \chi_{ \{ x+z \in \R, \, x'+z \notin \R \} }
\Xi_{\lambda}(x,z,t)  \, dz
\lesssim
\bigg( \frac{|x-x'|^2}{t} \bigg)^{\gamma},
\]
uniformly in $x,x' \in \R$ and $t>0$.
Moreover, one can take any $\gamma \in (0,1/2]$ satisfying
$\gamma \le \min_{1\le k \le n}(\lambda_k +1/2)$.
\end{lem}

\begin{proof}
Let $\gamma \in (0,1/2]$ satisfying
$\gamma \le \min_{1\le k \le n}(\lambda_k +1/2)$ be fixed.
Observe that if $|x - x'| > \sqrt{t}$, then using Lemma \ref{lem:intXi} and the obvious inequality $1 \le \Big( \frac{|x-x'|^2}{t} \Big)^{\gamma}$, we get the desired bound. Thus, we can assume that $|x - x'| \le \sqrt{t}$.
Since the constraint $x'+z \notin \R$ forces $z_k \le -x'_k$ for some $k \in \{ 1, \ldots, n \}$,
it is enough to verify that for every
$k \in \{ 1, \ldots, n \}$ we have
\[
\int_{|z|<\sqrt{t}} \chi_{ \{ x+z \in \R, \, z_k \le -x'_k \} } \Xi_{\lambda}(x,z,t)  \, dz
\lesssim
\frac{|x_k-x'_k|}{\sqrt{t}} +
\bigg( \frac{|x_k-x'_k|^2}{t} \bigg)^{\lambda_k + 1/2}.
\]
Further, by the product structure of $\Xi_\lambda(x,z,t)$, the one-dimensional version of Lemma \ref{lem:intXi} and \eqref{ball}
we see that it suffices to prove that
\begin{align*}
\int_{|z_k|<\sqrt{t}} \chi_{ \{  -x_k< z_k \le -x'_k \} }
\frac{(x_k+z_k)^{2\lambda_k}}{ \sqrt{t} (x_k+\sqrt{t})^{2\lambda_k} } \, dz_k
\lesssim
\frac{|x_k-x'_k|}{\sqrt{t}} + \bigg( \frac{|x_k-x'_k|^2}{t} \bigg)^{\lambda_k +1/2}.
\end{align*}
To proceed it is convenient to distinguish two cases.

\noindent \textbf{Case 1:} $\mathbf {x_k \ge 2\sqrt{t}.}$ We have $x_k + z_k \simeq x_k \simeq x_k + \sqrt{t}$ and the required bound follows.

\noindent \textbf{Case 2:} $\mathbf {x_k < 2\sqrt{t}.}$ Since
$x_k + \sqrt{t} \simeq \sqrt{t}$,
an integration gives us
\begin{align*}
\int_{|z_k|<\sqrt{t}} \chi_{ \{  -x_k< z_k \le -x'_k \} }
\frac{(x_k+z_k)^{2\lambda_k}}{ \sqrt{t} (x_k+\sqrt{t})^{2\lambda_k} } \, dz_k
\lesssim
\int_{-x_k}^{-x'_k} \frac{(x_k+z_k)^{2\lambda_k}}{ t^{\lambda_k + 1/2} } \, dz_k
\simeq
\bigg( \frac{|x_k-x'_k|^2}{t} \bigg)^{\lambda_k +1/2}.
\end{align*}
\end{proof}

In the proofs of Propositions~\ref{pro:kerest} and \ref{pro:kerestLAIP} we tacitly assume that passing with the differentiation in
$x_j,y_j$ or $t$ under integrals agains $dt, d\nu(t)$ or $du$ is legitimate. In fact such manipulations can easily be justified by means of the dominated convergence theorem and the estimates obtained in Lemma~\ref{lem:EST} and along the proofs of these propositions.

\begin{proof}[Proof of Proposition \ref{pro:kerest}; the case of $\mathcal{W}^{\lambda, \eta, +}(x,y)$]
The growth estimate \eqref{gr} is an easy consequence of \eqref{Bhk} and Lemma~\ref{lem:EST2}
(with $p=\infty$, $W=1$, $C=1/4$ and $|\alpha|=|\beta|=|\zeta|=|\rho|=|\tau|=u=0$).

For symmetry reasons we only need to show the smoothness condition \eqref{sm1}. In view of the
Mean Value Theorem we have
$$|\G_t^{\lambda,\eta,+}(x,y) - \G_t^{\lambda,\eta,+}(x',y)|
    \leq |x-x'| \Big| \nabla_{\!x} \G_t^{\lambda,\eta,+}(x,y) \big|_{x=\t} \Big|,$$
with $\t = \t(t,x,x',y)$ a convex combination of $x$ and $x'$. Hence, it remains to verify that
$$\bigg\| \Big| \nabla_{\!x} \G_t^{\lambda,\eta,+}(x,y) \big|_{x=\t} \Big| \bigg\|_{L^\infty (dt)}
    \lesssim \frac{1}{|x-y| w^+_{\lambda}(B(x,|x-y|))}, \qquad |x-y|>2|x-x'|.$$
This, however, follows by applying successively
Lemma~\ref{lem:EST} (taken with $K=|r|=|M|=0$ and $\ell=e_j$, $j=1, \dots, n$, where $e_j$ represents the $j$th coordinate vector in $\RR$), the relations
\begin{equation} \label{est1}
    \t \le x \vee x' , \qquad  |x - \t| \le |x-x'| , \qquad |x - x \vee x'| \le |x - x'|,
\end{equation}
Lemma~\ref{lem:theta} twice (with $z=\theta$ and $z=x \vee x'$),
Lemma~\ref{lem:EST2} (choosing $p=\infty$, $W=1$, $C=1/128$, $u=1$ and $\tau = 0$) and finally
Lemma~\ref{lem:double} (specified to $\gamma=1$ and $z=x \vee x'$).
\end{proof}

\begin{proof}[Proof of Proposition \ref{pro:kerest}; the case of $\mathcal{G}_{K,M}^{\lambda,\eta,+}(x,y)$]
The growth condition is obtained by combining Lemma~\ref{lem:EST} (taking $|\ell|=|r|=0$) and
Lemma~\ref{lem:EST2} (with $p=2$, $W=2K+|M|$, $C=1/8$, $u=0$ and $\tau = 0$).

We now prove the bound \eqref{sm1} for $\mathcal{G}_{K,M}^{\lambda,\eta,+}(x,y)$; the estimate \eqref{sm2} can be treated analogously.
After applying the Mean Value Theorem, our objective is to see that
$$ \bigg\| \Big| \nabla_{\! x} \partial_{t}^K \delta_{\eta,M,x} \G_t^{\lambda,\eta,+}(x,y) \big|_{x=\t} \Big| \bigg\|_{L^2 (t^{2K+|M|-1}dt)}
    \lesssim \frac{1}{|x-y|w^+_{\lambda}(B(x,|x-y|))}, \qquad |x-y|>2|x-x'|,$$
where $\t = \t(t,x,x',y)$ is a convex combination of $x$ and $x'$. Again, we use sequently
Lemma~\ref{lem:EST} (selecting $|r|=0$ and $\ell=e_j$, $j=1, \dots, n$),
relations \eqref{est1},
Lemma~\ref{lem:theta} twice (first with $z=\theta$ and secondly with $z=x \vee x'$),
Lemma~\ref{lem:EST2} (taken with $p=2$, $W=2K+|M|$, $C=1/128$, $u=1$ and $\tau = 0$) and
Lemma~\ref{lem:double} (with $\gamma=1$ and $z=x \vee x'$)
to get the desired estimate.
\end{proof}

\begin{proof}[Proof of Proposition \ref{pro:kerest}; the case of $K^{\lambda,\eta,+}_{\Phi}(x,y)$]
The growth condition is achieved directly by the fact that $\Phi$ is a bounded function,
Lemma~\ref{lem:EST} (taken with $K=1$ and $|\ell|=|r|=|M|=0$) and
Lemma~\ref{lem:EST2} (choosing $p=W=1$, $C=1/8$, $u=0$ and $\tau = 0$).

Next, we show the gradient estimate \eqref{grad}. Since $\Phi\in L^\infty(dt)$, our goal
is to obtain the bound
$$ \Big\| \big| \nabla_{\! x,y} \partial_{t} \G_t^{\lambda,\eta,+}(x,y) \big| \Big\|_{L^1 (dt)}
    \lesssim \frac{1}{|x-y|w^+_{\lambda}(B(x,|x-y|))}, \qquad x \ne y.$$
This, however, follows from
Lemma~\ref{lem:EST} (with $K=1$, $|M|=0$ and $|\ell|=0$, $r=e_j$ or $\ell=e_j$, $|r|=0$, $j=1, \dots,n$) and
Lemma~\ref{lem:EST2} (specified to $p=W=1$, $C=1/8$, $u=1$ and $\tau = 0$).
\end{proof}

\begin{proof}[Proof of Proposition \ref{pro:kerest}; the case of $K^{\lambda,\eta,+}_{\nu}(x,y)$]
Since $\nu$ is a complex measure, it has finite total variation, and then proving the
growth and smoothness properties for $K^{\lambda,\eta,+}_{\nu}(x,y)$ is reduced to
showing \eqref{gr} and \eqref{sm1} for the kernel $\{\G_t^{\lambda,\eta,+}(x,y)\}_{t>0}$
in the Banach space $\B=L^\infty(dt)$. This was already done in the case of
$\mathcal{W}^{\lambda, \eta, +}(x,y)$.
\end{proof}

\begin{proof}[Proof of Proposition \ref{pro:kerest}; the case of $R_M^{\lambda,\eta,+}(x,y)$]
Combining
Lemma~\ref{lem:EST} (taken with $K=|\ell|=|r|=0$) and
Lemma~\ref{lem:EST2} (choosing $p=1$, $W=|M|/2$, $C=1/8$, $u=0$ and $\tau = 0$)
we attain the growth bound for this kernel.

On the other hand, the gradient estimate is achieved by using
Lemma~\ref{lem:EST} (specified to  $K=0$ and $|\ell|=0$, $r=e_j$ or $\ell=e_j$, $|r|=0$, $j=1, \dots,n$) and
Lemma~\ref{lem:EST2} (with $p=1$, $W=|M|/2$, $C=1/8$, $u=1$ and $\tau =0$).
\end{proof}

\begin{proof}[Proof of Proposition \ref{pro:kerest}; the case of $\mathcal{S}^{\lambda,\eta,+}_{K,M}(x,y)$]
We first deal with the growth estimate.
Using Lemma~\ref{lem:EST}, and then Lemma~\ref{lem:qz} we infer that
\begin{align*}
& \left| \partial_{t}^K \delta_{\eta,M,\mathbf{x}} \G_{t}^{\lambda,\eta,+}(\mathbf{x},y)\Big|_{\mathbf{x} = x+z}
\right| \\
& \quad \lesssim
\sum_{\substack{ \eps, \zeta, \rho \in \{0,1\}^n \\
\alpha, \beta  \in \{0,1,2\}^n  }} (x+z)^{2\eps - \alpha\eps + \eta - \zeta\eta} y^{2\eps-\beta \eps + \eta - \rho \eta}
  t^{-(n/2+|\lambda|+ |\eta| + 2|\eps|)-K-|M|/2 + (|\alpha \eps| + |\zeta \eta| + |\beta \eps| + |\rho \eta|)/2} \\
& \quad \qquad \times
                \int \exp\left(-\frac{\q}{16t}\right)
d\Omega_{\lambda+\eta+ \mathbf{1} + \eps }(s),
\end{align*}
provided that $(z,t) \in A$.
Now an application of the estimate
\begin{equation}\label{est111}
|(x+z)^\kappa| \le (x+\sqrt{t}\mathbf{1})^\kappa
\lesssim
\sum_{0\le \tau \le \kappa} x^{\kappa - \tau} t^{|\tau|/2},
\qquad x\in\R, \quad (z,t) \in A,
\end{equation}
where $\kappa \in \mathbb{N}^n$ is fixed, gives the bound
\begin{align}\nonumber
& \left| \partial_{t}^K \delta_{\eta,M,\mathbf{x}} \G_{t}^{\lambda,\eta,+}(\mathbf{x},y)\Big|_{\mathbf{x} = x+z}
\right|  \\ \label{est222}
& \quad \lesssim
\sum_{\substack{ \eps, \zeta, \rho \in \{0,1\}^n \\ \nonumber
\alpha, \beta  \in \{0,1,2\}^n  }}
\sum_{0 \le \tau \le 2\eps - \alpha\eps + \eta - \zeta\eta}
x^{2\eps - \alpha\eps + \eta - \zeta\eta - \tau}
y^{2\eps-\beta \eps + \eta - \rho \eta} t^{-(n/2+|\lambda|+ |\eta| +2|\eps|)-K-|M|/2 } \\
& \quad \qquad \times
  t^{(|\alpha \eps| + |\zeta \eta| + |\beta \eps| + |\rho \eta| + |\tau|)/2}
                \int \exp\left(-\frac{\q}{16t}\right)
d\Omega_{\lambda+\eta+ \mathbf{1} + \eps }(s),
\end{align}
for $(z,t) \in A$. Using sequently this estimate, Lemma~\ref{lem:intXi} and then Lemma~\ref{lem:EST2} (taken with $p=2$, $W=2K + |M|$, $C = 1/16$, $u=0$) we arrive at the desired bound.

Next, we show the first smoothness estimate. More precisely, we will show \eqref{sm1} with any fixed $\gamma \in (0,1/2]$ satisfying
$\gamma < \min_{1\le k \le n} (\lambda_k +1/2)$. To proceed, it is natural to split the region of integration $A$ into four subsets, depending on whether $x+z,x'+z$ are in $\R$ or not. We define
\begin{displaymath}
\begin {array}{lll}
A_1 &= \left\{ (z,t) \in A : x+z \in \R, x'+z \in \R \right\},
\quad
A_2 &= \left\{ (z,t) \in A : x+z \in \R, x'+z \notin \R \right\},\\
A_3 &= \left\{ (z,t) \in A : x+z \notin \R, x'+z \in \R \right\},
\quad
A_4 &= \left\{ (z,t) \in A : x+z \notin \R, x'+z \notin \R \right\}.
\end {array}
\end{displaymath}
The analysis related to $A_4$ is trivial and the case of $A_3$ is analogous to $A_2$, thus we analyze only two cases.

\noindent
{\bf Case 1:} \textbf{The norm related to}
$\mathbf{L^2(A_{1},t^{2K + |M| - 1}dzdt)}.$
By the triangle inequality
\begin{align*}
& \left| \partial_{t}^K \delta_{\eta,M,\mathbf{x}} \G_{t}^{\lambda,\eta,+}(\mathbf{x},y)\Big|_{\mathbf{x} = x+z}
\sqrt{\Xi_{\lambda}(x,z,t)}
-
\partial_{t}^K \delta_{\eta,M,\mathbf{x}} \G_{t}^{\lambda,\eta,+}(\mathbf{x},y)\Big|_{\mathbf{x} = x'+z}
\sqrt{\Xi_{\lambda}(x',z,t)}
\right| \\
& \quad \le
\left| \partial_{t}^K \delta_{\eta,M,\mathbf{x}} \G_{t}^{\lambda,\eta,+}(\mathbf{x},y)\Big|_{\mathbf{x} = x+z}
-
\partial_{t}^K \delta_{\eta,M,\mathbf{x}} \G_{t}^{\lambda,\eta,+}(\mathbf{x},y)\Big|_{\mathbf{x} = x'+z}
\right| \sqrt{\Xi_{\lambda}(x',z,t)} \\
& \qquad  +
\left| \partial_{t}^K \delta_{\eta,M,\mathbf{x}} \G_{t}^{\lambda,\eta,+}(\mathbf{x},y)\Big|_{\mathbf{x} = x+z}
\right|
\left| \sqrt{\Xi_{\lambda}(x,z,t)} - \sqrt{\Xi_{\lambda}(x',z,t)} \right| \\
& \quad \equiv
I_1(x,x',y,z,t) + I_2(x,x',y,z,t).
\end{align*}
We will treat $I_1$ and $I_2$ separately. An application of the Mean Value Theorem and then successively Lemma~\ref{lem:EST}, Lemma~\ref{lem:qz}, \eqref{est1}, Lemma~\ref{lem:theta} twice (first with $z=\theta$ and then with $z= x \vee x'$) and \eqref{est111} gives
\begin{align*}
I_1(x,x',y,z,t) & \lesssim
|x - x'| \sum_{\substack{ \eps, \zeta, \rho \in \{0,1\}^n \\
\alpha, \beta  \in \{0,1,2\}^n  }}
\sum_{0 \le \tau \le 2\eps - \alpha\eps + \eta - \zeta\eta}
(x \vee x')^{2\eps - \alpha\eps + \eta - \zeta\eta - \tau}
y^{2\eps-\beta \eps + \eta - \rho \eta} \\
& \qquad \times
  t^{-(n/2+|\lambda|+|\eta| +2|\eps|)-K-|M|/2 -1/2 + (|\alpha \eps| + |\zeta \eta| + |\beta \eps| + |\rho \eta| + |\tau|)/2} \\
& \qquad \times
                \int \exp\left(-\frac{q(x \vee x', y ,s)}{256t}\right)
d\Omega_{\lambda+\eta+ \mathbf{1} + \eps }(s)
\sqrt{\Xi_{\lambda}(x',z,t)},
\end{align*}
provided that $(z,t) \in A_1$ and $|x-y|>2|x-x'|$.
Combining this with Lemma~\ref{lem:intXi},
Lemma~\ref{lem:EST2} (specified to $p=2$, $W=2K + |M|$, $C=1/256$, $u=1$)
and Lemma~\ref{lem:double} (with $\gamma=1$ and $z=x \vee x'$) leads to the required bound for $I_1$. We now focus on $I_2$.
Using the estimate \eqref{est222} and Lemma~\ref{lem:intdifXi} we get
\begin{align}\label{est444}
& \left\| I_2(x,x',y,z,t)  \right\|_{L^2(A_1,t^{2K + |M| -1} dzdt)} \\ \nonumber
& \quad \lesssim
|x - x'|^{\gamma}
\sum_{\substack{ \eps, \zeta, \rho \in \{0,1\}^n \\
\alpha, \beta  \in \{0,1,2\}^n  }}
\sum_{0 \le \tau \le 2\eps - \alpha\eps + \eta - \zeta\eta}
x^{2\eps - \alpha\eps + \eta - \zeta\eta - \tau}
y^{2\eps-\beta \eps + \eta - \rho \eta} \\ \nonumber
& \qquad \times
\Big\| t^{-(n/2+|\lambda|+|\eta| +2|\eps|)
-K-|M|/2 -\gamma/2 + (|\alpha \eps| + |\zeta \eta| + |\beta \eps| + |\rho \eta| + |\tau|)/2} \\ \nonumber
& \qquad \times
                \int \exp\left(-\frac{\q}{16t}\right)
d\Omega_{\lambda+\eta+ \mathbf{1} + \eps }(s) \Big\|_{L^2(t^{2K + |M| -1} dt)}.
\end{align}
This, with the aid of Lemma~\ref{lem:EST2} (taken with $p=2$, $W=2K + |M|$, $C=1/16$, $u=\gamma$), completes the analysis associated with $A_1$.

\noindent
{\bf Case 2:} \textbf{The norm related to}
$\mathbf{L^2(A_{2},t^{2K + |M| - 1}dzdt)}.$
Since $\mathcal{S}^{\lambda,\eta,+}_{K,M}(x',y)=0$, our task is to show that
\begin{align}\label{est333}
\|
\mathcal{S}^{\lambda,\eta,+}_{K,M}(x,y)
\|_{L^2(A_2,t^{2K + |M| -1} dzdt)}
\lesssim \bigg(\frac{|x-x'|}{|x-y|} \bigg)^{\gamma}\, \frac{1}{w^+_{\lambda}(B(x,|x-y|))},
\end{align}
for $|x-y|>2|x-x'|$.
By means of the estimate \eqref{est222} and Lemma~\ref{lem:intXi2} we see that
\begin{align*}
\|
& \mathcal{S}^{\lambda,\eta,+}_{K,M}(x,y)
\|_{L^2(A_2,t^{2K + |M| -1} dzdt)} \\
& \quad
\lesssim |x-x'|^{\gamma}
\sum_{\substack{ \eps, \zeta, \rho \in \{0,1\}^n \\
\alpha, \beta  \in \{0,1,2\}^n  }}
\sum_{0 \le \tau \le 2\eps - \alpha\eps + \eta - \zeta\eta}
x^{2\eps - \alpha\eps + \eta - \zeta\eta - \tau}
y^{2\eps-\beta \eps + \eta - \rho \eta}
\Big\| t^{-(n/2+|\lambda|+|\eta| +2|\eps|) } \\
& \qquad \times
t^{-K-|M|/2 -\gamma/2 + (|\alpha \eps| + |\zeta \eta| + |\beta \eps| + |\rho \eta| + |\tau|)/2}
                \int \exp\left(-\frac{\q}{16t}\right)
d\Omega_{\lambda+\eta+ \mathbf{1} + \eps }(s) \Big\|_{L^2(t^{2K + |M| -1} dt)}.
\end{align*}
The right-hand side here coincides with the right-hand side of \eqref{est444}, and \eqref{est333} follows.

We now focus on the second smoothness condition \eqref{sm2}. We prove it with $\gamma=1$.
Using the Mean Value Theorem, and then sequently
Lemma~\ref{lem:EST}, Lemma~\ref{lem:qz},
Lemma~\ref{lem:theta} twice (with $z=\theta$ and $z= y \vee y'$) and \eqref{est111}
we obtain
\begin{align*}
& \left|
\partial_{t}^K \delta_{\eta,M,\mathbf{x}} \G_{t}^{\lambda,\eta,+}(\mathbf{x},y)\Big|_{\mathbf{x} = x+z}
-  \partial_{t}^K \delta_{\eta,M,\mathbf{x}} \G_{t}^{\lambda,\eta,+}(\mathbf{x},y')\Big|_{\mathbf{x} = x+z}
\right|
\sqrt{\Xi_{\lambda}(x,z,t)} \chi_{ \{ x+z \in \R \} }
\\
& \quad \lesssim
|y - y'|
\sum_{\substack{ \eps, \zeta, \rho \in \{0,1\}^n \\
\alpha, \beta  \in \{0,1,2\}^n  }}
\sum_{0 \le \tau \le 2\eps - \alpha\eps + \eta - \zeta\eta}
x^{2\eps - \alpha\eps + \eta - \zeta\eta - \tau}
(y \vee y')^{2\eps-\beta \eps + \eta - \rho \eta}
t^{-(n/2+|\lambda|+|\eta| +2|\eps|) -K} \\
& \qquad \times
t^{-|M|/2 -1/2 + (|\alpha \eps| + |\zeta \eta| + |\beta \eps| + |\rho \eta| + |\tau|)/2}
                \int \exp\left(-\frac{q(x,y \vee y',s)}{256t}\right)
d\Omega_{\lambda+\eta+ \mathbf{1} + \eps }(s) \\
& \qquad \times
\sqrt{\Xi_{\lambda}(x,z,t)}
\chi_{ \{ x+z \in \R \} },
\end{align*}
for $(z,t) \in A$ and $|x-y| > 2|y-y'|$.
An application of Lemma~\ref{lem:intXi},
Lemma~\ref{lem:EST2} (specified to $p=2$, $W=2K + |M|$, $C=1/256$ $u=1$)
and then Lemma~\ref{lem:double} leads directly to the desired estimate.

This finishes the whole reasoning justifying Proposition \ref{pro:kerest}.
\end{proof}


To prove Proposition~\ref{pro:kerestLAIP}, we need
Fa\`a di Bruno's formula for the $K$th derivative, $K \ge 1$, of the composition
of two functions (see \cite{Jo} for related references and interesting historical remarks),
\begin{equation} \label{Faa}
\partial_t^K(g\circ f)(t) = \sum \frac{K!}{k_1! \cdot \ldots \cdot k_K!} \;
    \big(\partial^{k_1+\ldots+k_K} g\big)
    \circ f(t) \bigg( \frac{\partial_t^1 f(t)}{1!}\bigg)^{k_1}\cdot \ldots \cdot
    \bigg( \frac{\partial_t^K f(t)}{K!}\bigg)^{k_K},
\end{equation}
where the summation runs over all $k_1,\ldots,k_K \ge 0$ such that $k_1+2k_2+\ldots+K k_K = K$.

\begin{proof}[Proof of Proposition \ref{pro:kerestLAIP}]
The reasoning is based on the subordination formula \eqref{subprinc} and a careful repetition of the arguments from the proof of Proposition \ref{pro:kerest} (the case of $\mathcal{S}^{\lambda,\eta,+}_{K,M}(x,y)$).
We give the details only in the case of the growth condition \eqref{gr}, leaving the proof of the smoothness bounds to the reader.

By \eqref{subprinc} and Fa\`a di Bruno's formula \eqref{Faa} applied with 
$g(r) = \delta_{\eta,M,\mathbf{x}} \G_{r}^{\lambda,\eta,+}(\mathbf{x},y)\Big|_{\mathbf{x} = x+z}$
and 
$f(t) = t^2/(4u)$, we get
\begin{align*}
& \partial_{t}^K \delta_{\eta,M,\mathbf{x}} \mathbb{P}_{t}^{\lambda,\eta,+}(\mathbf{x},y)\Big|_{\mathbf{x} = x+z} \\
& \qquad =
\sum_{k_1 + 2k_2 = K} c_{k_1,k_2} \int_0^\infty
\partial_r^{k_1 + k_2} \delta_{\eta,M,\mathbf{x}} \G_{r}^{\lambda,\eta,+}(\mathbf{x},y)
\Big|_{\mathbf{x} = x+z }^{r=t^2/(4u)}
t^{k_1} u^{-k_1 - k_2} \frac{e^{-u} \, du}{\sqrt{u}},
\end{align*}
where $c_{k_1,k_2}$ are constants; observe that this formula works also for $K=0$.
Using Minkowski's integral inequality we may reduce our task to showing that
\begin{align*}
& \int_0^\infty \left\|
\partial_r^{k_1 + k_2} \delta_{\eta,M,\mathbf{x}} \G_{r}^{\lambda,\eta,+}(\mathbf{x},y)
\Big|_{\mathbf{x} = x+z }^{r=t^2/(4u)}
t^{k_1}  \sqrt{\Xi_{\lambda}(x,z,t^2)}
\,\chi_{\{x+z\in\R\}}
\right\|_{ L^2(\Gamma,t^{2K + 2|M| - 1}dzdt) } \\
& \qquad \times u^{-k_1 - k_2}
\frac{e^{-u} \, du}{\sqrt{u}}
\lesssim
 \frac{1}{w^+_{\lambda}(B(x,|x-y|))}, \qquad x \ne y,
\end{align*}
where $k_1, k_2 \in \N$ are fixed and such that $k_1 + 2k_2 = K$.
After the change of variable $t^2 \mapsto t$, we see that the left-hand side above is, up to a constant factor, equal to
\begin{align*}
I &\equiv
\int_0^\infty \left\|
\partial_r^{k_1 + k_2} \delta_{\eta,M,\mathbf{x}} \G_{r}^{\lambda,\eta,+}(\mathbf{x},y)
\Big|_{\mathbf{x} = x+z }^{r=t/(4u)}
  \sqrt{\Xi_{\lambda}(x,z,t)}
\,\chi_{\{x+z\in\R\}}
\right\|_{ L^2(A,t^{k_1 + K + |M| - 1}dzdt) } \\
& \qquad \qquad \times
u^{-k_1 - k_2} \frac{e^{-u} \, du}{\sqrt{u}}.
\end{align*}
Proceeding in a similar way as at the beginning of the proof of Proposition~\ref{pro:kerest} (the case of $\mathcal{S}^{\lambda,\eta,+}_{K,M}(x,y)$), namely using Lemma~\ref{lem:EST}, Lemma~\ref{lem:qz} and then \eqref{est111}, we obtain
\begin{align*}
& \left| \partial_{r}^{k_1 + k_2} \delta_{\eta,M,\mathbf{x}} \G_{r}^{\lambda,\eta,+}(\mathbf{x},y)\Big|_{\mathbf{x} = x+z}^{r=t/(4u)}
\right| \\
& \quad \lesssim
\sum_{\substack{ \eps, \zeta, \rho \in \{0,1\}^n \\
\alpha, \beta  \in \{0,1,2\}^n  }}
\sum_{0 \le \tau \le 2\eps - \alpha\eps + \eta - \zeta\eta}
x^{2\eps - \alpha\eps + \eta - \zeta\eta - \tau}
y^{2\eps-\beta \eps + \eta - \rho \eta} (t/u)^{-(n/2+|\lambda|+|\eta| +2|\eps|)
- k_1 - k_2 - |M|/2 } \\
& \quad \qquad \times
  (t/u)^{(|\alpha \eps| + |\zeta \eta| + |\beta \eps| + |\rho \eta| )/2} t^{|\tau|/2} e^{u/2}
                \int \exp\left(-\frac{\q}{4t} u \right)
d\Omega_{\lambda+\eta+ \mathbf{1} + \eps }(s), \qquad
(z,t) \in A.
\end{align*}
Since the expression in the right-hand side is independent of $z$, an application of Lemma~\ref{lem:intXi} and then the change of variable $t/u \mapsto t$ gives
\begin{align*}
I &\lesssim
\sum_{\substack{ \eps, \zeta, \rho \in \{0,1\}^n \\
\alpha, \beta  \in \{0,1,2\}^n  }}
\sum_{0 \le \tau \le 2\eps - \alpha\eps + \eta - \zeta\eta}
\int_0^\infty \Big\|
x^{2\eps - \alpha\eps + \eta - \zeta\eta - \tau}
y^{2\eps-\beta \eps + \eta - \rho \eta}
t^{-(n/2+|\lambda|+|\eta| +2|\eps|)
- k_1 - k_2 - |M|/2} \\
&  \quad \times
t^{ (|\alpha \eps| + |\zeta \eta| + |\beta \eps| + |\rho \eta| + |\tau| )/2}
                \int \exp\left(-\frac{\q}{4t}  \right)
d\Omega_{\lambda+\eta+ \mathbf{1} + \eps }(s)
\Big\|_{L^2(t^{k_1 + K + |M| - 1}dt) }
u^{|M|/2 + |\tau|/2} \frac{e^{-u/2} \, du}{\sqrt{u}}.
\end{align*}
This together with Lemma~\ref{lem:EST2} (taken with $p=2$, $W=2k_1 + 2k_2 + |M|$, $C = 1/4$, $u=0$) leads directly to the required bound.
\end{proof}


\end{document}